\documentclass[final,leqno,onetabnum]{siamltex704}
\usepackage{amsmath,amssymb,epsfig,subfigure,color,graphicx,url}
\usepackage{exscale}
\usepackage{epstopdf}
\usepackage{multirow}
\usepackage{tikz,pgf,pgfarrows,pgfnodes,pgfautomata,pgfheaps,pgfshade}
\usepackage{caption,tikz-qtree}
\usepackage{float}
\usetikzlibrary{matrix,arrows}
\usepackage{algorithmicx,algorithm}
\usepackage{rotating,booktabs}
\usepackage{enumerate}
\allowdisplaybreaks

\usepackage[colorlinks=true]{hyperref}
\hypersetup{urlcolor=red,citecolor=red,linkcolor=red}

\newtheorem{remark}[theorem]{{\it Remark\/}}
\newtheorem{example}[theorem]{{\it Example\/}}
\renewcommand{\thetable}{\thesection{}.\arabic{table}}

\numberwithin{algorithm}{section}
\numberwithin{table}{section}
\numberwithin{figure}{section}

\usepackage[titletoc,title]{appendix}
\newcommand{\ba}{{\bf a}}
\newcommand{\bb}{{\bf b}}
\newcommand{\bc}{{\bf c}}

\newcommand{\br}{{\bf r}}

\newcommand{\bx}{{\bf x}}
\newcommand{\by}{{\bf y}}
\newcommand{\bz}{{\bf z}}
\newcommand{\bu}{{\bf u}}
\newcommand{\bv}{{\bf v}}

\newcommand{\cb}{{\cal B}}
\newcommand{\cc}{{\cal C}}
\newcommand{\ci}{{\cal I}}
\newcommand{\cj}{{\cal J}}

\newcommand{\ch}{{\cal H}}

\newcommand{\ct}{{\cal T}}

\newcommand{\R}{{\mathbb{R}}}
\newcommand{\Rm}{{\mathbb{R}^m}}
\newcommand{\Rn}{{\mathbb{R}^n}}

\newcommand{\Rmm}{{\mathbb{R}^{m\times m}}}
\newcommand{\Rmn}{{\mathbb{R}^{m\times n}}}

\newcommand{\dist}{{\rm dist}}
\newcommand{\dom}{{\rm dom}}

\newcommand{\BE}{\begin{equation}}
\newcommand{\EE}{\end{equation}}

\DeclareMathOperator*{\argmin}{argmin}
\newcommand{\normmm}[1]{{\vert\kern-0.25ex \vert\kern-0.25ex \vert #1
    \vert\kern-0.25ex \vert\kern-0.25ex\vert}}

\begin{document}
\title{An augmented Lagrangian method with exact multipliers for non-separable  composite $\ell_0$-$\ell_2$ regularization\thanks{The work is supported by the National Natural Science Foundation of China under grant Nos. 12401496, 12501525, the Fundamental Research Funds for the Central Universities under grant No. D5000240300, the Natural Science Foundation of Jiangxi Province with No. 20244BCE52256, and Jiangxi Science and Technology Normal University Doctoral Research Initiation Fund No. 2023BSQD25.
}}
\author{Huan Ren\thanks{School of Mathematical Sciences, Jiangxi Science and Technology Normal University, Nanchang, 330038, People's Republic of China (jxrh1994@163.com).}\and Guiyun Xiao\thanks{Corresponding author. School of Mathematics and Statistics, Northwestern Polytechnical University, Xi'an, 710072, People's Republic of China (xiaogy@nwpu.edu.cn and xiaogy999@163.com). 
} }
\date{}
\maketitle
\begin{abstract}
This paper studies a non-separable  composite $\ell_0$-$\ell_2$ regularization model that simultaneously enforces sparsity and smoothness for inverse problems. The $\ell_0$ norm induces inherent nonconvexity and nonsmoothness, while linear transformations further introduce nonseparability, making the problem computationally challenging to solve. The existing inexact augmented Lagrangian method suffers from high computational complexity and unstable convergence. To overcome these difficulties, we develop two novel augmented Lagrangian algorithms with exact multipliers, designed respectively for the full row-rank case and the general matrix case, where all subproblems are globally optimized via closed-form solutions.  Furthermore, we prove linear convergence of the proposed method when the transformation matrix is full row rank. In the general setting, all accumulation points of the generated sequence are KKT points for the original problem. Numerical experiments on synthetic data, trend filtering, and image smoothing demonstrate the superior efficiency and accuracy of the proposed methods over the existing method, confirming our theoretical analysis.
\end{abstract}

\vspace{3mm}
{\bf Keywords.} Augmented Lagrangian method, exact multipliers, non-separable  composite $\ell_0$-$\ell_2$ regularization.

\vspace{3mm}
{\bf AMS subject classifications.} 65F22, 65K05, 90C26.

\section{Introduction}\label{sec1}
With the growing interest in inverse problem research \cite{BB98,EL12,NW01}, the importance of regularization techniques has become increasingly prominent. Inverse problems are typically ill-posed, and regularization techniques can alleviate their ill-conditioning while preventing overfitting. This has also greatly advanced the research and development of regularization techniques, including total variation (TV)-based regularization \cite{AV94}, sparsity prior-based regularization\cite{D92,D06}, Tikhonov regularization \cite{H10}, and Bayesian prior model-based regularization \cite{LS09}.

Among sparsity-prior regularizers, $\ell_0$-regularization is the most intuitive and effective across signal estimation \cite{AG21}, feature coding \cite{C18}, radar space-time adaptive processing \cite{LW24}, image processing \cite{AS15,DZ13,YG19}, and machine learning \cite{CZ11}. Unfortunately, the non-convexity of the $\ell_0$-norm leads to NP-hard computation. Researchers thus widely adopt convex $\ell_1$-regularization \cite{CW08,CD98,ZL12}, though it attenuates critical features and degrades practical performance \cite{FL01,Z10}. 
Based on this consideration, in terms of sparsity regularization, we still consider the most direct model, which is the $\ell_0$-regularization model.

In the other hand, the discontinuity of the $\ell_0$-norm may cause the model's solution to lose smoothness to a certain extent. Similar to sparsity, in numerous specific application scenarios, smoothness is a crucial property that characterizes the natural structure of solutions in the corresponding scene. Based on this, in order to balance sparsity and smoothness, Song et al. proposed non-separable  composite $\ell_0$-$\ell_2$ regularization model, which takes the following form:
\BE\label{intro}
\begin{array}{lc}
\min\limits_{\bu\in\Rn}  &  \displaystyle f(\bu)+\lambda\|W\bu\|_0+\frac{\beta}{2}\|W\bu\|^2, \\[2mm]
 \end{array}
\EE
where $\lambda$, $\beta>0$, $f(\bu)$ is a data fidelity term, $W:\R^n \mapsto\R^m$ is a predefined transform.

When $W=I$, \eqref{intro} reduces to the classical $\ell_0$ and $\ell_2$ regularization model with abundant existing solvers. Under the restricted isometry property (RIP), typical \(\ell_0\)-constrained least-squares solvers fall into greedy and thresholding categories. Greedy algorithms cover matching pursuit \cite{MZ93}, orthogonal matching pursuit \cite{DM94,TG07}, CoSaMP \cite{NT09} and subspace pursuit \cite{DM09}; thresholding methods include hard thresholding \cite{B12,BD08-2,BD09,F11}, soft thresholding \cite{D95,E06}, and heavy-ball hard thresholding pursuit \cite{SZ23}. Cai et al. \cite{CJ22} proposed a randomized alternating minimization algorithm for solving the sparse phase retrieval problem. When the measurement matrix satisfies the RIP, Zhao \cite{Z20} introduced the optimal $k$-thresholding algorithm, which was followed by a series of subsequent improved algorithms, including heavy-ball-type algorithms.

%

In model \eqref{intro}, the $\ell_0$-norm penalty acts on $W\bu$ rather than $\bu$. An invertible square W allows rewriting the model as a standard $\ell_0$-regularized problem penalizing $\by=W\bu$, solvable via existing approaches. However, in practical scenarios, $W$ is often non-square and non-invertible, invalidating conventional solvers. This core challenge is that while the $\ell_0$ proximal operator of $\bu$ has a closed-form, the counterpart for $\|W\bu\|_0$ is nonconvex, nonsmooth and lacks explicit solutions, greatly complicating optimization. To address this issue, Song et al. \cite{SL20} recently proposed an inexact augmented Lagrangian method for solving model \eqref{intro}. For solving the involved subproblem, they adopted an alternating scheme. However, this alternating scheme results in extremely high computational complexity for solving the subproblems, thereby severely impairing the efficiency of the entire algorithm. In terms of theoretical analysis, they established the convergence analysis of the algorithm for problems involving non-separable  composite $\ell_0$-$\ell_2$ regularization terms. Specifically, if the sequence generated by the algorithm is bounded, then there exists an accumulation point that is a KKT point of the equivalent problem, i.e., a local minimizer of the original problem.

In this paper, we focus on solving the non-separable  composite $\ell_0$-$\ell_2$ regularization model that arises in various inverse problems. Given the drawbacks of existing inexact augmented Lagrangian methods, together with intrinsic challenges including the non-convexity and non-smoothness of the $\ell_0$ norm as well as non-separability induced by linear transformation, we develop two novel augmented Lagrangian algorithms with exact multipliers. The first algorithm converges linearly to a KKT point for full row-rank $W$, while the second guarantees that every accumulation point is a KKT point for arbitrary $W$. A notable advantage shared by both algorithms is that all subproblems admit closed-form solutions. We further conduct comprehensive numerical experiments on synthetic datasets and practical real-world applications such as trend filtering and image smoothing to validate the performance of our methods. The experimental results demonstrate that our algorithms outperform the existing inexact augmented Lagrangian method in terms of accuracy and efficiency.

Throughout this paper, we use the following notation. Let $\Rmn$ denote the set of all $m\times n$ real matrices and $\Rn=\R^{n\times 1}$ the set of $n$-dimensional real vectors. The space  $\Rn$ is equipped with the Euclidean inner product $\langle \cdot,\cdot\rangle$ and the corresponding induced norm $\|\cdot\|$ (also denoting the spectral norm of a matrix). The superscript $\top$ represents the transpose of a vector or matrix. For a matrix $B$ and an index set $\mathcal{I}$, $B_{\mathcal{I}}$ stands for the submatrix consisting of the columns of $B$ indexed by $\mathcal{I}$.  The $\ell_0$-norm $\|\cdot\|_0$ counts the number of nonzero entries in a vector. The symbol $\|\cdot\|_\infty$ denotes the vector infinity norm. For any  $\bc=(c_1,\ldots,c_n)^\top\in\Rn$, $\supp(\bc):=\{j\in [n]\; | \; c_j\neq 0\}$, where $[n]:=\{1,\ldots,n\}$. Given a vector $\by$ and an index set $\mathcal{J}$, $\by_{\mathcal{J}}$ represents the subvector of $\by$ formed by the entries of $\by$ indexed by $\mathcal{J}$.
%

The rest of this paper is organized as follows. In Section \ref{sec:2}, we propose two augmented Lagrangian methods with exact multipliers to solve problem \eqref{intro}, which correspond to two different scenarios: one for the case where $W$ is full row rank and the other for the general case with an arbitrary matrix $W$. In Section \ref{sec:3}, we derive explicit solutions for all subproblems and establish convergence properties of the two proposed algorithms under mild assumptions. In Section \ref{sec:4},  we report some numerical tests to indicate the effectiveness of our methods. Finally, some concluding remarks are given in Section \ref{sec:5}.

\section{Augmented Lagrangian method with exact multipliers}\label{sec:2}
In this section, we will propose an augmented Lagrangian method with exact multipliers to solve problem \eqref{intro}, considering both the case where $W$ is full row rank and the general case where $W$ is an arbitrary matrix.

As is well known, for non-convex and non-smooth problems, there exists a close relationship between optimal points and KKT points. We often aim to find KKT points as a pathway to locating the optimal solution of such problems. 
First, we reformulate model \eqref{intro} into the following form by introducing an auxiliary variable:
\BE\label{equ:reformu}
\begin{aligned}
\min\limits_{\bu,\by}  & \quad  f(\bu)+\lambda\|\by\|_0+\frac{\beta}{2}\|W\bu\|^2, \\
\mbox{s.t.} &\quad \by=W\bu.
 \end{aligned}
\EE
The KKT conditions of the above problem is: a point $(\bu,\by)$ is called to a KKT point of problem \eqref{equ:reformu} if there exists a vector $\bz$ such that
\BE\label{KKTfir}
{\bf 0}=\nabla f(\bu)+\beta W^{\top}W\bu+W^{\top}\bz,
\EE
\BE\label{KKTsec}
{\bf 0}\in \partial \lambda\|\by\|_0-\bz,
\EE
\BE\label{KKTthi}
W\bu-\by={\bf 0}.
\EE
In this paper, we aim to find $\bu$, $\by$ and $\bz$ that satisfy the above equations. A classical infeasible method for solving such problems is the augmented Lagrangian method.
Given a parameter $\rho$, the augmented Lagrangian function for problem \eqref{equ:reformu} is
\begin{eqnarray}
L(\bu,\by,\bz)=f(\bu)+\langle W\bu-\by , \bz\rangle+\lambda\|\by\|_0 +\frac{\beta}{2}\|W\bu\|^2+\frac{\rho}{2}\|W\bu-\by\|^2.
\end{eqnarray}

Considering that \eqref{equ:reformu} is non-convex and non-smooth, the classical augmented Lagrangian method is not applicable. In what follows, we modify and enhance the classical augmented Lagrangian method to adapt it to different scenarios of problem \eqref{equ:reformu} corresponding to different $W$.

\subsection{Augmented Lagrangian method with exact multipliers for full-row-rank $W$}

Based on the equation \eqref{KKTfir} satisfied by the KKT points problem \eqref{equ:reformu}, we can treat $\bz$ as a function of the variable $\bu$. It can be observed that when $W$ is of full row rank, 
$\bz$ can be uniquely expressed as a function of $\bu$, i.e.,
\BE\label{equ:zk}
\bz(\bu)=(W^{\top})^{\dagger}(-\nabla f(\bu)-\beta W^{\top}W\bu ).
\EE
This implies that when $\bu$, $\bz$ and $\by$ satisfy the KKT conditions, $\bz$ and $\bu$ must adhere to the functional relationship $\bz = \bz(\bu)$. In other words, once $\bu$ is determined and satisfies the KKT conditions, $\bz$ is unique determined accordingly.

Based on this idea, when updating the multiplier $\bz$, we abandon the classical gradient-based update method for the multiplier $\bz$. Instead, we update $\bz$ using the functional value $\bz(\bu)$ corresponding to the updated $\bu$. Moreover, since we treat $\bz$ as a function of $\bu$, when evaluating the quality of iterative values and conducting convergence proofs, we adopt the following correspondingly defined function
\begin{equation}\label{hfullrank}
h(\bu,\by)=f(\bu)+\big\langle W\bu-\by , \bz(\bu)\big\rangle+\lambda\|\by\|_0 +\frac{\beta}{2}\|W\bu\|^2+\frac{\rho}{2}\|W\bu-\by\|^2
\end{equation}
instead of the original augmented Lagrangian function $L(\bu,\by,\bz)$ for subsequent convergence analysis.

Building on the above idea, we propose an augmented Lagrangian method with exact multipliers for full-row-rank $W$ to solve model \eqref{equ:reformu}, detailed in Algorithm \ref{algor:rlmem}.

\begin{algorithm}[!ht]
\caption{Augmented Lagrangian method with exact multipliers for full row rank  $W$} \label{algor:rlmem}
\begin{description}
\item [{\rm Step 0.}] Input $\lambda$ and $\beta$. Choose $\bu^0$, $\by^0=W\bu^0$, $\rho>0$. Let $k:=0$. Compute
    \[\bz^0=\bz(\bu^0)=(W^{\top})^{\dagger}(-\nabla f(\bu^0)-\beta W^{\top}W\bu^0 ).\]
\item [{\rm Step 1.}] Compute
\BE\label{rquesu}
{\bu}^{k+1}=\argmin_{\bu\in\Rn} \Big\{f(\bu)+\big\langle W\bu-\by^k , \bz^k\big\rangle +\frac{\beta}{2}\|W\bu\|^2+\frac{t_k}{2}\|\bu-\bu^k\|^2 +\frac{\rho}{2}\|W\bu-\by^k\|^2\Big\}.
\EE
\item [{\rm Step 2.}] Take $\bz^{k+1}=\bz(\bu^{k+1})=(W^{\top})^{\dagger}(-\nabla f(\bu^{k+1})-\beta W^{\top}W\bu^{k+1} )$ and compute
\BE\label{rquesy}
{\by}^{k+1}=\argmin_{\by\in\Rm} \Big\{\lambda \|\by\|_0+\big\langle W\bu^{k+1}-\by ,\bz^{k+1} \big\rangle +\frac{t_k}{2}\|\by-\by^k\|^2+\frac{\rho}{2}\|W\bu^{k+1}-\by\|^2\Big\}.
\EE
\item [{\rm Step 3.}]  Replace $k$ by $k+1$ and go to  Step 1.
\end{description}
\end{algorithm}

\subsection{Augmented Lagrangian method with exact multipliers for general case of $W$}

Next, we extend the idea of the augmented Lagrangian method with exact multipliers to the general case of $W$. However, when $W$ is not full row rank, we find that the solution to the equation
\eqref{KKTfir} is not unique. Additionally, $\bz$ must satisfy \eqref{KKTsec} with ${\bf 0}=\br -\bz$, where $\br \in \partial \lambda\|\by\|_0$.

According to existing results in literature \cite{SL20}, it is known that
\[
\partial \lambda \|\by\|_0=(r_1, \ldots, r_m)^{\top}\quad \mbox{with}\quad
r_i=
\left\{
\begin{array}{ll}
\{ 0 \},\quad &\mbox{if}~ y_i\neq 0,\\
\mathbb{R},\quad& \mbox{if}~y_i=0.
\end{array}
\right.
\]
Therefore, when $y_i \neq 0$, $z_i=0$ and $y_i = 0$, $z_i\in\R$. Let
\BE\label{equ:z}
\left\{
\begin{array}{ll}
(\bz(\bu,\by))_{\mathcal{C}}=((W^{\top})_{\mathcal{C}})^\dagger(-\nabla f(\bu)-\beta W^{\top}W\bu),&\quad \mathcal{C}=\{i | y_i=0, i\in [m]\},\\
(\bz(\bu,\by))_{\mathcal{I}}={\bf 0},&\quad \mathcal{I}=\{i | y_i\neq 0, i\in [m]\}.
\end{array}
\right.
\EE
When $(\bu,\by)$ is a KKT point, it can be confirmed that $\bz=\bz(\bu,\by)$ is a solution satisfying equations \eqref{KKTfir} and \eqref{KKTsec}. This implies that if $(\bu,\by)$ is either an optimal solution or a KKT points, then $(\bu, \by, \bz)$ satisfies the KKT conditions with $\bz=\bz(\bu,\by)$.

Based on the definition of $\bz(\bu,\by)$, it is not only related to $\bu$, but also depends on the positions of the non-zero elements of $\by$. Specifically, $\bz(\bu,\by)$ is a discontinuous, step function with respect to $\by$. Owing to the modified definition of the multiplier function $\bz(\bu,\by)$, the function characterizing the quality of the sequences generated by the algorithm is adjusted accordingly to:
\[
h(\bu,\by)=f(\bu)+\langle W\bu-\by , \bz(\bu,\by)\rangle+\lambda\|\by\|_0 +\frac{\beta}{2}\|W\bu\|^2+\frac{\rho}{2}\|W\bu-\by\|^2,
\]
where $\bz(\bu,\by)$ follows equation \eqref{equ:z}. However, if we still use $\bz=\bz(\bu,\by)$ to update the multiplier, we need to propose an algorithm that generates sequences to overcome the issue of jump points, which may cause non-monotonic decrease in the $h(\bu,\by)$. We divide the update of the sequence and the corresponding function values into two steps: (1) Compute update the value of $\bu$ to obtain the corresponding new function value $h(\bu^{k+1},\by^k)$; (2) Compute update the value of $\by$ to obtain the corresponding new function value $h(\bu^{k+1},\by^{k+1})$. We observe that when $\by$ is fixed, $\bz(\bu,\by)$ is continuous with respect to $\bu$, the function $h(\bu, \cdot)$ is continuous. Therefore, using a simple proximal point method for $\bu$ can ensure a monotonic decrease in the updated function value. When $\bu$ is fixed, the function $h(\bu,\by)$ exhibits jump discontinuities. Consequently, to ensure a monotonic decrease in the function value, it is necessary to minimize $h(\bu,\by)$ with fixed $\bu$ to update $\by$.

\begin{algorithm}[!ht]
	\caption{Augmented Lagrangian method with exact multipliers for general case of $W$} \label{algor:clmem}
	\begin{description}
		\item [{\rm Step 0.}] Input $\lambda$ and $\beta$. Choose $\bu^0$, $\by^0=W\bu^0$, $\rho>0$, $t_k>0$. Let $k:=0$.
		\item [{\rm Step 1.}] Take $\mathcal{C}^k=\{i | y_i^k=0, i\in [n]\}$,  $\mathcal{I}^k=\{i | y_i^k\neq 0, i\in [n]\}$, $(\bz(\bu^k,\by^k))_{\mathcal{C}^k}=((W^{\top})_{\mathcal{C}^k})^\dagger(-\nabla f(\bu^k)-\beta W^{\top}W\bu^k)$ and $(\bz(\bu^k,\by^k))_{\mathcal{I}^k}={\bf 0}$.
		\begin{align}\label{cquesu}
			{\bu}^{k+1} = \argmin_{\bu\in\Rm} \Big\{&f(\bu)+\big\langle W\bu-\by^k , \bz(\bu^k,\by^k) \big\rangle +\frac{\beta}{2}\|W\bu\|^2+\frac{t_k}{2}\|\bu-\bu^k\|^2\nonumber\\
			&+\frac{\rho}{2}\|W\bu-\by^k\|^2\Big\}.
		\end{align}
		\item  [{\rm Step 2.}] Compute
		\BE\label{cquesy}
		{\by}^{k+1}=\argmin_{\by\in\Rn} h(\bu^{k+1},\by).
		\EE
		\item [{\rm Step 3.}]  Replace $k$ by $k+1$ and go to  Step 1.
	\end{description}
\end{algorithm}

Based on the above analysis, we propose an augmented Lagrangian method with exact multipliers for general case of $W$ to solve problem \eqref{equ:reformu}, detailed in Algorithm \ref{algor:clmem}.


\section{Convergence analysis}\label{sec:3}

In this section, we will separately present the convergence analysis of the augmented Lagrangian method with exact multipliers for different scenarios of the matrix $W$.  

\subsection{Solution of subproblems}
Since the specific form of $f$ is crucial to the expression of its solution in the subproblems of Algorithms \ref{algor:rlmem}--\ref{algor:clmem}, particularly for the $\bu$-subproblem. Therefore, we have investigated different choices of data fidelity terms in various inverse problems, such as $f=\frac12\|\bu-\bu_0\|^2$ in image smoothing and $f=\frac12\|A\bu-\bu_0\|^2$ in CT image reconstruction. Considering the generality of the problem, we set $f=\frac12\|A\bu-\bu_0\|^2$ and derive the solution expression form for the subproblem accordingly, along with the subsequent convergence analysis.

We summarize the explicit expression of the globally optimal solutions to the subproblems defined in \eqref{rquesu} and \eqref{rquesy} of Algorithm \ref{algor:rlmem} as Theorem \ref{thesolu}.

\begin{theorem}\label{thesolu}
Let $\bu^k,\by^k$ be the current iterate generated by Algorithm \ref{algor:rlmem}, then the global optimal solution $\bu^{k+1}$ of subproblem \eqref{rquesu} and a global optimal solution of subproblem \eqref{rquesy} are
\BE\label{solutionru}
\bu^{k+1}=(A^{\top}A+(\beta+\rho) W^{\top}W+t_kI)^{-1}(A^{\top}\bu_0-W^{\top}\bz^k+\rho W^{\top}\by^k+t_k\bu^k)
\EE
and
\BE\label{solutionry}
\by^{k+1}=\ch_{\sqrt{2\lambda/(\rho+t_k)}}(\frac{\rho}{\rho+t_k}W\bu^{k+1}+\frac{t_k}{\rho+t_k}\by^k+\frac{1}{\rho+t_k}\bz^{k+1}),
\EE
where $\mathcal{H}$ is the hard thresholding operator, i.e., $\mathcal{H}_\lambda(s)=s$ if $|s|>\lambda$, $\mathcal{H}_\lambda(s)=0$ if $|s|\leq\lambda$.
\end{theorem}

\begin{proof}
By the optimal condition of \eqref{rquesu}, $\bu^{k+1}$ satisfies
\[
A^{\top}(A\bu^{k+1}-\bu_0)+W^{\top}\bz^k+\beta W^{\top}W\bu^{k+1}+t_k(\bu^{k+1}-\bu^k)+\rho W^{\top}(W\bu^{k+1}-\by^k)={\bf 0},
\]
which implies \eqref{solutionru}.

Next, by \eqref{rquesy}, we have
\begin{eqnarray*}
&& \argmin_{\by\in\Rm} \{\lambda \|\by\|_0+\langle W\bu^{k+1}-\by, \bz^{k+1} \rangle +\frac{t_k}{2}\|\by-\by^k\|^2+\frac{\rho}{2}\|W\bu^{k+1}-\by\|^2\}\\
&=& \argmin_{\by\in\Rm} \{\lambda \|\by\|_0 +\frac{\rho+t_k}{2}\|\by-\frac{\rho}{\rho+t_k}W\bu^{k+1}-\frac{t_k}{\rho+t_k}\by^k-\frac{1}{\rho+t_k}\bz^{k+1}\|^2\}
\end{eqnarray*}
and therefore $\by^{k+1}=\ch_{\sqrt{2\lambda/(\rho+t_k)}}(\frac{\rho}{\rho+t_k}W\bu^{k+1}+\frac{t_k}{\rho+t_k}\by^k+\frac{1}{\rho+t_k}\bz^{k+1})$ is a global optimal solution of subproblem \eqref{rquesy}.
\end{proof}

For the explicit solutions to subproblems \eqref{cquesu}--\eqref{cquesy} in Algorithm \ref{algor:clmem}, we have organized them as follows:
\begin{theorem}\label{solutionguy}
Let $\{\bu^k,\by^k\}$ be a bounded sequence generated by Algorithm \ref{algor:clmem}. Define 
\BE\label{equ:M}
M:=\sup \{\|((W^{\top})_{\mathcal{C}})^\dagger(-\nabla f(\bu^k)-\beta W^{\top}W\bu^k)\|_{\infty}, ~\forall~\cc \subset [m], \; k=1,2, \ldots \}.
\EE
Suppose $\rho>\frac{m M^2}{\lambda}$. For a sufficiently large $k$, assume that no component of sequence $\{|W\bu^k|\}$ lie within the small interval $\Big(\frac{-M+\sqrt{M^2+\rho \lambda/m}}{\rho}, \frac{M+\sqrt{M^2+2\rho (m+1) \lambda}}{\rho}\Big)$. Then,\\
(i) the global optimal solution $\bu^{k+1}$ of subproblem \eqref{cquesu} is
\BE\label{equ:uk}
\bu^{k+1}=(t_kI+(\beta+\rho)W^{\top}W+A^{\top}A)^{-1}(A^{\top}\bu_0-W^{\top}\bz(\bu^k,\by^k)+t_k\bu^k+\rho W^{\top}\by^k),
\EE
where
\BE\label{equ:zuk}
\left\{
\begin{array}{ll}
(\bz(\bu^k,\by^k))_{\cc^k}=((W^{\top})_{\mathcal{C}^k})^\dagger(-\nabla f(\bu^k)-\beta W^{\top}W\bu^k),&\quad \mathcal{C}^k=\{i | y_i^k=0, i\in [m]\},\\
(\bz(\bu^k,\by^k))_{\ci^k}={\bf 0},&\quad \mathcal{I}^k=\{i | y_i^k\neq0, i\in [m]\};
\end{array}
\right.
\EE
(ii) a global optimal solution $\by^{k+1}$ of subproblem \eqref{cquesy} is
\BE\label{equ:yk}
y_i^{k+1}=\left\{
\begin{array}{ll}
(W\bu^{k+1})_i, & \mbox{if}~~ |(W\bu^{k+1})_i|\ge \frac{M+\sqrt{M^2+2\rho (m+1) \lambda}}{\rho},\\[2mm]
0, & \mbox{if}~~|(W\bu^{k+1})_i|\le \frac{-M+\sqrt{M^2+\rho \lambda/m}}{\rho}.
\end{array}
\right.
\EE
\end{theorem}
\begin{proof}
(i) Since subproblem \eqref{cquesu} is a strongly convex problem, it has exactly one unique optimal solution. From the optimality conditions of \eqref{cquesu}, we have
\[
A^{\top}(A\bu^{k+1}-\bu_0)+W^{\top}\bz(\bu^k,\by^k)+\beta W^{\top} W\bu^{k+1}+t_k(\bu^{k+1}-\bu^k)+\rho W^{\top}(W\bu^{k+1}-\by^k)={\bf 0}.
\]
We can readily derive the solution to the above equation that satisfies \eqref{equ:uk}, where $\bz(\bu^k,\by^k)$ defined as in \eqref{equ:zuk}.

(ii) For subproblem \eqref{cquesy}, we have
\begin{eqnarray*}
&&{\by}^{k+1}=\argmin_{\by\in\Rn} h(\bu^{k+1},\by) \\
\Leftrightarrow~~  &&{\by}^{k+1}=\argmin_{\by\in\Rn} \{ \lambda \|\by\|_0+\big\langle W\bu^{k+1}-\by , \bz(\bu^{k+1},\by) \big\rangle +\frac{\rho}{2}\|W\bu^{k+1}-\by\|^2\}.
\end{eqnarray*}
Let $\cj_1=\{j\in[m] |\; |(W\bu^{k+1})_j|\ge \frac{M+\sqrt{M^2+2\rho (m+1) \lambda}}{\rho}\}$, $\cj_2=\{j\in[m] |\; |(W\bu^{k+1})_j|\le \frac{-M+\sqrt{M^2+\rho \lambda/m}}{\rho}\}$, and $\cj_1\cup\cj_2=[m]$.
Pick an arbitrary $\by\in\R^m$, define $
\phi_i=\lambda \|\by_i\|_0+(W\bu^{k+1}-\by)_i(\bz(\bu^{k+1},\by))_i +\frac{\rho}{2}(W\bu^{k+1}-\by)_i^2$, let $\ct_1=\{i\in [m]| y_i\neq 0\}$, $\ct_2=\{i\in [m]| y_i= 0\}$, then $\ct_1\cup\ct_2=[m]$ and we have
\begin{align*}
&\lambda \|\by\|_0+\big\langle W\bu^{k+1}-\by , \bz(\bu^{k+1},\by) \big\rangle +\frac{\rho}{2}\|W\bu^{k+1}-\by\|^2\nonumber\\
=& \sum_{i\in\cj_1\cap \ct_1}\phi_i
+\sum_{i\in\cj_1\cap \ct_2}\phi_i
+\sum_{i\in\cj_2\cap\ct_1}\phi_i
+ \sum_{i\in\cj_2\cap \ct_2}\phi_i\nonumber\\
=& \sum_{i\in\cj_1\cap \ct_1}\left[\lambda+\frac{\rho}{2}(W\bu^{k+1}-\by)_i^2\right]+\sum_{i\in\cj_1\cap \ct_2}\left[(W\bu^{k+1})_i(\bz(\bu^{k+1},\by))_i +\frac{\rho}{2}(W\bu^{k+1})_i^2\right]\nonumber\\
&+\sum_{i\in\cj_2\cap\ct_1}\left[\lambda+\frac{\rho}{2}(W\bu^{k+1}-\by)_i^2\right]+ \sum_{i\in\cj_2\cap \ct_2}\left[(W\bu^{k+1})_i(\bz(\bu^{k+1},\by))_i +\frac{\rho}{2}(W\bu^{k+1})_i^2\right]\nonumber\\
\ge& \underbrace{\sum_{i\in\cj_1\cap \ct_1}\lambda}_{A_1}+\underbrace{\sum_{i\in\cj_1\cap \ct_2}\left[-|(W\bu^{k+1})_i|M +\frac{\rho}{2}(W\bu^{k+1})_i^2\right]}_{A_2}\nonumber\\
&+\underbrace{\sum_{i\in\cj_2\cap\ct_1}\lambda}_{A_3}+ \underbrace{\sum_{i\in\cj_2\cap \ct_2}\left[-|(W\bu^{k+1})_i|M +\frac{\rho}{2}(W\bu^{k+1})_i^2\right]}_{A_4}, \nonumber
\end{align*}
where the inequality follows from equation \eqref{equ:M}. Then, for $A_4$, we have
\begin{eqnarray*}
A_4&=&
\sum_{i\in\cj_2\cap \ct_2} \left[\frac{\rho}{2}(W\bu^{k+1})_i^2-|(W\bu^{k+1})_i|M\right]
\ge \sum_{i\in\cj_2\cap \ct_2} \left(\frac{\rho}{2}\frac{M^2}{\rho^2}-\frac{M}{\rho}M\right)\\
&=&\sum_{i\in\cj_2\cap \ct_2} -\frac{M^2}{2\rho}\ge
\sum_{i\in\cj_2\cap \ct_2} - \frac{\lambda}{2m},
\end{eqnarray*}
where the first inequality follows from the fact that the quadratic function $\frac\rho2x^2-M|x|$ attains its minimum at $x=\frac M\rho$, the last inequality follows from the assumption that $\rho>\frac{m M^2}{\lambda}$. Let $|\ct_1|=t_1$, $|\ct_2|=t_2$, $|\cj_1|=s_1$, $|\cj_2|=s_2$, and $|\cj_1\cap \ct_1|=p_1$. Then $t_1+t_2=s_1+s_2$ and
\begin{eqnarray*}
&&A_1+A_2+A_3+A_4\\
&\geq&\sum_{i\in\cj_1\cap \ct_1}\lambda+\sum_{i\in\cj_1\cap \ct_2}(-|(W\bu^{k+1})_i|M +\frac{\rho}{2}(W\bu^{k+1})_i^2)+\sum_{i\in\cj_2\cap\ct_1}\lambda+ \sum_{i\in\cj_2\cap \ct_2}- \frac{\lambda}{2m}\\
&\ge& p_1\lambda+(s_1-p_1)(m+1)\lambda
+(t_1-p_1)\lambda+(-\frac{\lambda}{2m})(s_2+p_1-t_1),\\
\end{eqnarray*}
where the last inequality follows from the fact that $\min\{\frac\rho2x^2-M|x|\}=(m+1)\lambda$ in $|x|\geq\frac{M+\sqrt{M^2+2\rho (m+1) \lambda}}{\rho}$.

Similarly, if we take by $\tilde{\by}=\by^{k+1}$ and its components are shown in \eqref{equ:yk}, denote $
\tilde{\phi}_i=\lambda \|\tilde{\by}_i\|_0+(W\bu^{k+1}-\tilde{\by})_i(\bz(\bu^{k+1},\tilde{\by}))_i +\frac{\rho}{2}(W\bu^{k+1}-\tilde{\by})_i^2$, then we have
\begin{align*}
&\lambda \|\tilde{\by}\|_0+\big\langle W\bu^{k+1}-\tilde{\by} , z(\bu^{k+1},\tilde{\by}) \big\rangle +\frac{\rho}{2}\|W\bu^{k+1}-\tilde{\by}\|^2\\
=& \sum_{i\in\cj_1\cap \ct_1}\tilde{\phi}_i+\sum_{i\in\cj_1\cap \ct_2}\tilde{\phi}_i
+\sum_{i\in\cj_2\cap\ct_1}\tilde{\phi}_i
+ \sum_{i\in\cj_2\cap \ct_2}\tilde{\phi}_i\\
\le& \sum_{i\in\cj_1\cap \ct_1}\lambda+\sum_{i\in\cj_1\cap \ct_2}\lambda+\sum_{i\in\cj_2\cap\ct_1}\left[|(W\bu^{k+1})_i|M +\frac{\rho}{2}(W\bu^{k+1})_i^2\right]\\
&+ \sum_{i\in\cj_2\cap \ct_2}\left[|(W\bu^{k+1})_i|M+\frac{\rho}{2}(W\bu^{k+1})_i^2\right]\\
\le& p_1\lambda+(s_1-p_1)\lambda+(t_1-p_1)\frac{\lambda}{2m} +(s_2+p_1-t_1)\frac{\lambda}{2m},
\end{align*}
where the last inequality follows from the fact that $\max\{\frac\rho2x^2-M|x|\}=\frac{\lambda}{2m}$ in $|x|\le \frac{-M+\sqrt{M^2+\rho \lambda/m}}{\rho}$.

We consider the difference between the function values of \eqref{cquesy} at any $\by$ and $\tilde{\by}$ given by \eqref{equ:yk},
\begin{align}
&h(\bu^{k+1},\by)-h(\bu^{k+1},\tilde{\by})\nonumber\\
=&\lambda \|\by\|_0+\big\langle W\bu^{k+1}-\by , \bz(\bu^{k+1},\by) \big\rangle +\frac{\rho}{2}\|W\bu^{k+1}-\by\|^2\nonumber\\
&-(\lambda \|\tilde{\by}\|_0+\big\langle W\bu^{k+1}-\tilde{\by} , z(\bu^{k+1},\tilde{\by}) \big\rangle +\frac{\rho}{2}\|W\bu^{k+1}-\tilde{\by}\|^2)\nonumber\\
\ge& p_1\lambda+(s_1-p_1)(m+1)\lambda+(t_1-p_1)\lambda+(-\frac{\lambda}{2m})(s_2+p_1-t_1)\nonumber\\
&-( p_1\lambda+(s_1-p_1)\lambda+(t_1-p_1)\frac{\lambda}{2m}+(s_2+p_1-t_1)\frac{\lambda}{2m})\nonumber\\
=&(s_1-p_1)m\lambda+(t_1-p_1)\lambda-(2s_2+p_1-t_1)\frac{\lambda}{2m}.\label{equ:s1p1}
\end{align}

Next, we prove that the necessary condition for $\by$ to be the minimum point of subproblem \eqref{cquesy} is $\cj_1=\ct_1$ and $\cj_2=\ct_2$, i.e. $s_1=p_1=t_1$.

First, we divide the relationship between $s_1$ and $p_1$ into two cases: $s_1 - p_1 \ge 1$ and $s_1 = p_1$. When $s_1 - p_1 \ge 1$, we have
\begin{equation}
h(\bu^{k+1},\by)-h(\bu^{k+1},\tilde{\by})\ge m\lambda+(t_1-p_1)\lambda(1+\frac{1}{2m})-\frac{\lambda}{m}s_2
\ge m\lambda-\frac{\lambda}{m}s_2>0,\nonumber
\end{equation}
where the second inequality follows from $t_1-p_1\ge0$, and the last inequality holds by virtue of $s_2<m$. From the above equation, it can be observed that when $s_1 - p_1 \ge 1$, $h(\bu^{k+1},\by) - h(\bu^{k+1},\tilde{\by}) > 0$. This implies that the minimum point cannot be attained in the case of $s_1 - p_1 \ge 1$, so the minimum point must satisfy $s_1 = p_1$, i.e., $\cj_1\cap\ct_2=\varnothing$.

On the basis of $s_1=p_1$, we also divide the relationship between $p_1$ and $t_1$ into two cases: $t_1-p_1\ge1$ and $t_1=p_1$. When $t_1-p_1\ge1$, then $|\cj_2\cap\ct_2|=s_2+p_1-t_1=s_2+s_1-t_1=m-t_1=t_2=|\ct_2|$, from which we readily derive that
\begin{eqnarray*}
&&h(\bu^{k+1},\by)-h(\bu^{k+1},\tilde{\by})\\
&\ge&p_1\lambda+(t_1-p_1)\lambda+(-\frac{\lambda}{2m})t_2-( p_1\lambda+(t_1-p_1)\frac{\lambda}{2m}+t_2\frac{\lambda}{2m})\\
&=&(t_1-p_1)\lambda-  t_2\frac{\lambda}{2m}-(m-p_1)\frac{\lambda}{2m}= (t_1-p_1)\lambda-(m+t_2-p_1)\frac{\lambda}{2m}>0,
\end{eqnarray*}
where the last inequality follows from $t_1-p_1\ge1$ and $m+t_2-p_1<2m$. Similarly, this also indicates that the minimum point cannot be attained when $t_1 - p_1 \ge 1$; thus, the minimum point must satisfy $t_1 = p_1$, i.e., $\cj_2\cap\ct_1=\varnothing$. Combining the result $s_1 = p_1$, we readily obtain $s_1 = p_1 = t_1$ and $s_2 = t_2$; that is to say, $\mathcal{J}_1 = \mathcal{T}_1$ and $\mathcal{J}_2 = \mathcal{T}_2$.

%

To summarize, a necessary condition for $\by$ to be a minimum point is
\[
\mathcal{J}_1 = \mathcal{T}_1 \quad\mbox{and}\quad\mathcal{J}_2 = \mathcal{T}_2.
\]
In other words, when $|(W\bu^{k+1})_i|\ge (M+\sqrt{M^2+2\rho (m+1) \lambda})/{\rho}$, the $i$-th of the optimal solution of subproblem \eqref{cquesy} at the $(k+1)$-th step is non-zero, i.e., $y_i^{k+1}\neq0$; when $|(W\bu^{k+1})_i|\le (-M+\sqrt{M^2+\rho \lambda/m})/\rho$, the $i$-th of the optimal solution subproblem \eqref{cquesy} at the $(k+1)$-th step is zero, i.e. $y_i^{k+1}=0$.

The above has determined the positions where the components of the optimal solutions of the subproblems \eqref{cquesy} are non-zero. Next, we determine the corresponding value of the components when the $i$-th component of the optimal solutions of the subproblems \eqref{cquesy} are non-zero.

By the \eqref{cquesy}, we have
\BE\label{findsolutiony}
{\by}^{k+1}=\argmin_{\by\in\Rn}\{\lambda \|\by\|_0+\big\langle W\bu^{k+1}-\by , \bz(\bu^{k+1},\by) \big\rangle +\frac{\rho}{2}\|W\bu^{k+1}-\by\|^2\}.
\EE
If $y_i^{k+1}\neq 0$, we have $(\bz(\bu^{k+1},\by))_i=0$. Then
\begin{eqnarray*}
&&\big\langle W\bu^{k+1}-\by^{k+1} , \bz(\bu^{k+1},\by^{k+1}) \big\rangle\\
&=&\big\langle W\bu^{k+1},  \bz(\bu^{k+1},\by^{k+1}) \big\rangle-\big\langle \by^{k+1}, \bz(\bu^{k+1},\by^{k+1}) \big\rangle\\
&=&\big\langle W\bu^{k+1},  \bz(\bu^{k+1},\by^{k+1}) \big\rangle
=\big\langle (W\bu^{k+1})_{\cc^{k+1}},  (\bz(\bu^{k+1},\by^{k+1}))_{\cc^{k+1}}\big\rangle\\
&=&\big\langle (W\bu^{k+1})_{\cc^{k+1}},  
((W^\top)_{\cc^{k+1}})^\dag(-\nabla f(\bu^{k+1})-\beta W^\top W\bu^{k+1})\big\rangle
\end{eqnarray*}
Therefore, for any $\by\in\R^m$, 
\begin{eqnarray*}
&&\lambda \|\by\|_0+\big\langle W\bu^{k+1}-\by , \bz(\bu^{k+1},\by) \big\rangle +\frac{\rho}{2}\|W\bu^{k+1}-\by\|^2\\
&=& \sum_{i\in\ci}\big(\lambda+\frac{\rho}{2}(y_i-(W\bu^{k+1})_i)^2\big) \\
&&+\sum_{i\in\cc}(W\bu^{k+1})_i[((W^\top)_i)^\dag(-\nabla f(\bu^{k+1})-\beta W^\top W\bu^{k+1})]+\sum_{i\in\cc}\frac\rho2(W\bu^{k+1})_i^2,
\end{eqnarray*}
where $\mathcal{C}=\{i | y_i=0, i\in [m]\}$ and $\mathcal{I}=\{i | y_i\neq0, i\in [m]\}$. Obviously, the objective function of \eqref{findsolutiony} attains its minimum if and only if $y_i=(W\bu^{k+1})_i$ for $i\in\ci$. Specifically, when $|(W\bu^{k+1})_i|\ge (M+\sqrt{M^2+2\rho (m+1) \lambda})/{\rho}$, $y_i$ must equal to $(W\bu^{k+1})_i$.   This thus proves that $\by^{k+1}$ is a globally optimal solution to subproblem \eqref{cquesy} at $(k+1)$-th step.
\end{proof}

\begin{remark}
In Theorem \ref{solutionguy}, for all sufficiently large $k$, we assume that no component of sequence $\{|W\bu^k|\}$ lie within the small interval $((-M+\sqrt{M^2+\rho \lambda/m})/\rho, (M$ $+\sqrt{M^2+2\rho (m+1) \lambda})/\rho)$. This may appear unusual, but in practice, it is not difficult to satisfy because when $\rho$ is sufficiently large,
\begin{eqnarray*}
&&\frac{M+\sqrt{M^2+2\rho (m+1) \lambda}}{\rho}-\frac{-M+\sqrt{M^2+\rho \lambda/m}}{\rho}\\
&=&\frac{2M}{\rho}+\frac{\sqrt{M^2+2\rho (m+1) \lambda}-\sqrt{M^2+\rho \lambda/m}}{\rho}\\
&=&\frac{2M}{\rho}+\frac{2(m+1)\lambda-\lambda/m}{\sqrt{M^2+2\rho (m+1) \lambda}+\sqrt{M^2+\rho \lambda/m}},
\end{eqnarray*}
the length of above interval tends to zero. This indicates that as the regularization parameter $\rho$ increases, the length of the small interval continuously decreases, making the assumption that $|(W\bu^k)_i|$ does not lie within this small interval a mild condition.
\end{remark}

\subsection{Convergence analysis of Algorithm \ref{algor:rlmem}}

In this subsection, we demonstrate the linear convergence results of the proposed augmented Lagrangian method with exact multipliers for row-full-rank $W$ under mild conditions.

In the following lemma, we show that the constraint violation at the current iterate can be bounded by the norm of difference between the updated value and the current iterate, which facilitates the subsequent proof of monotonic decrease in the function value.

\begin{lemma}\label{lem:3.4}
Let $\{(\bu^k,\by^k)\}$ be the sequence generated by Algorithm \ref{algor:rlmem}, then
\[
\left\|\bu^{k+1}-\bu^k\right\|\ge\frac{\rho}{\left\|A^{\top} A+(\beta +\rho)W^{\top} W+t_kI\right\|\left\|W^{\dagger}\right\|^2} \left\|W^{\dagger}(W\bu^k-\by^k)\right\|.
\]
\end{lemma}
\begin{proof}
Let $W=U\Sigma_mV_m^{\top}$ be the singular value decomposition of $W$, where $U,\Sigma_m\in \Rmm$ and $V_m\in\mathbb{R}^{n\times m}$.  Then
\begin{eqnarray}
\left\|W^{\dagger}(W\bu^k-\by^k)\right\|&=& \left\|V_m\Sigma_m^{-1}U^{\top}(W\bu^k-\by^k)\right\| =\left\|\Sigma_m^{-1}U^{\top}(W\bu^k-\by^k)\right\|\nonumber\\
&\le& \left\|\Sigma_m^{-1}\right\|_2\left\|U^{\top}(W\bu^k-\by^k)\right\|\nonumber\\
&=&\left\|W^{\dagger}\right\|_2\left\|UU^{\top}(W\bu^k-\by^k)\right\|\nonumber\\
&=&\left\|W^{\dagger}\right\|_2\left\|(W^{\top})^{\dagger}W^{\top}(W\bu^k-\by^k)\right\|\label{wwandw}.
\end{eqnarray}
By \eqref{solutionru} in Theorem \ref{thesolu} and $ \bz^k=(W^{\top})^{\dagger}(-\nabla f(\bu^{k})-\beta W^{\top}W\bu^{k} )=(W^{\top})^{\dagger}(-A^{\top}(A\bu^k-\bu_0)-\beta W^{\top}W\bu^k )$, we have
\begin{eqnarray*}
&& \left\|\bu^{k+1}-\bu^k\right\|\\
&=&\left\|(A^{\top} A+(\beta+\rho) W^{\top}W+t_k I)^{-1}(A^{\top} \bu_0-W^{\top} \bz^k+\rho W^{\top} \by^k+t_k\bu^k)-\bu^k\right\| \\
&=&\left\|(A^{\top} A+(\beta+\rho) W^{\top}W+t_k I)^{-1}(A^{\top} \bu_0-W^{\top} \bz^k+\rho W^{\top} \by^k+t_k\bu^k-A^{\top} A \bu^k\right.\\
&&\left.-(\beta+\rho) W^{\top} W\bu^k-t_k\bu^k)\right\| \\
& =&\left\|(A^{\top} A+(\beta+\rho) W^{\top}W+t_k I)^{-1}(A^{\top} \bu_0-W^{\top} \bz^k+\rho W^{\top} \by^k-A^{\top} A \bu^k\right.\\
&&-\left.(\beta+\rho) W^{\top} W\bu^k)\right\|\\
& \overset{(*)}{\geq}& \frac{\left\|A^{\top} \bu_0-W^{\top} \bz^k+\rho W^{\top} \by^k-A^{\top} A \bu^k-(\beta+\rho) W^{\top} W\bu^k\right\|}{\left\|A^{\top} A+(\beta +\rho)W^{\top} W+t_kI\right\|_2} \\
&=&\frac{\left\|-W^{\top}\bz^k+(-A^{\top} (A \bu^k-\bu_0)-\beta W^{\top} W\bu^k)-\rho W^{\top}( W\bu^k-\by^k)\right\|}{\left.\|A^{\top} A+(\beta +\rho)W^{\top} W+t_kI\right\|_2} \\
&=&\frac{\|(I-W^{\top}(W^{\top})^{\dagger})\big(-\nabla f(\bu^{k})-\beta W^{\top}W\bu^k\big)-\rho W^{\top}(W\bu^k-\by^k)\|}{\left\|A^{\top} A+(\beta +\rho)W^{\top} W+t_kI\right\|_2}\\
&\overset{(**)}{\ge}&\frac{ \left\|(W^{\top})^{\dagger}\big((I-W^{\top}(W^{\top})^{\dagger}) \big(-\nabla f(\bu^{k})-\beta W^{\top}W\bu^k\big)-\rho W^{\top}(W\bu^k-\by^k)\big)\right\|}{\left\|A^{\top} A+(\beta +\rho)W^{\top} W+t_kI\right\|_2\left\|(W^{\top})^{\dagger}\right\|_2} \\
&=&\frac{\rho}{\left\|A^{\top} A+(\beta +\rho)W^{\top} W+t_kI\right\|_2\left\|(W^{\top})^{\dagger}\right\|_2}
\left\|(W^{\top})^{\dagger}W^{\top}(W\bu^k-\by^k)\right\| \\
&\ge& \frac{\rho}{\left\|A^{\top} A+(\beta +\rho)W^{\top} W+t_kI\right\|_2\left\|W^{\dagger}\right\|_2^2} \left\|W^{\dagger}(W\bu^k-\by^k)\right\|,
\end{eqnarray*}
where the inequality $(*)$ follows by the fact that for any $\bx$, $\by$ satisfying $\bx = D^{-1}\by$, we have $\|\bx\| \ge \frac{1}{\|D\|_2} \|\by\|$, the inequality $(**)$ follows by $\|(W^\top)^\dag\bx\|\le\|(W^\top)^\dag\|_2\|\bx\|$, the last inequality is follows from \eqref{wwandw} and $\|(W^{\top})^{\dagger}\|_2=\|W^{\dagger}\|_2$.
\end{proof}

The following theorem characterizes the relationship between the norm of the difference between two consecutive iterates and the corresponding difference in function values.

\begin{lemma}\label{hfunbound}
Let $\{(\bu^k,\by^k)\}$ be the sequence generated by Algorithm \ref{algor:rlmem}. Let $\rho\ge \max\{t,1\}$ and $t=\max_{k}\{t_k\}$,
\BE\label{equ:tk}
\begin{aligned}
t_k \ge& \max\{ 4( \left\|W^{\top}(W^{\top})^{\dagger}(A^{\top}A+\beta W^{\top}W)\right\|+ \left\|W^{\top}W\right\| \left\|W^{\dagger}\right\|^2 \left\|A^{\top}A+\beta W^{\top}W\right\|\\
&+\Big(\left\|A^{\top}A+\beta W^{\top}W\right\| + 1\Big) \left\|W^{\dagger}\right\|^2 \left\|A^{\top}A+\beta W^{\top}W\right\| ),~c\},
\end{aligned}
\EE
where $c>0$, then
\[
h(\bu^k,\by^k)-h(\bu^{k+1},\by^{k+1})\ge \frac{c}{4}\left(\left\|\bu^{k+1}-\bu^k\right\|^2 +\left\|\by^{k+1}-\by^k\right\|^2\right).
\]
\end{lemma}

\begin{proof}
By \eqref{rquesu} and \eqref{rquesy}, we have
\begin{equation}\label{inequru}
\begin{aligned}
& f(\bu^k) + \frac{\beta}{2}\left\|W\bu^k\right\|^2 + \left\langle W\bu^k - \by^k, \bz^k \right\rangle + \frac{\rho}{2}\left\|W\bu^k - \by^k\right\|^2 
\ge  f(\bu^{k+1}) + \frac{\beta}{2}\left\|W\bu^{k+1}\right\|^2 \\
&+ \left\langle W\bu^{k+1} - \by^k, \bz^{k} \right\rangle + \frac{\rho}{2}\left\|W\bu^{k+1} - \by^k\right\|^2 + \frac{t_k}{2}\left\|\bu^{k+1} - \bu^k\right\|^2
\end{aligned}
\end{equation}
and
\begin{equation}\label{inequry}
\begin{aligned}
& \lambda \left\|\by^k\right\|_0 + \left\langle W\bu^{k+1} - \by^k , \bz^{k+1} \right\rangle + \frac{\rho}{2}\left\|W\bu^{k+1} - \by^k\right\|^2 \\
&\ge  \lambda \left\|\by^{k+1}\right\|_0 + \left\langle W\bu^{k+1} - \by^{k+1} , \bz^{k+1} \right\rangle + \frac{t_k}{2}\left\|\by^{k+1} - \by^k\right\|^2 + \frac{\rho}{2}\left\|W\bu^{k+1} - \by^{k+1}\right\|^2.
\end{aligned}
\end{equation}
By leveraging \eqref{inequru} and \eqref{inequry}, we obtain
\begin{eqnarray*}
&& h(\bu^k,\by^k) - h(\bu^{k+1},\by^{k+1}) \\
&=& h(\bu^k,\by^k) - h(\bu^{k+1},\by^k) + h(\bu^{k+1},\by^k) - h(\bu^{k+1},\by^{k+1}) \\
&\ge& \frac{t_k}{2}\left\|\bu^{k+1} - \bu^k\right\|^2 + \left\langle W\bu^{k+1} - \by^k, \bz^{k} \right\rangle - \left\langle W\bu^{k+1} - \by^k, \bz^{k+1} \right\rangle + \frac{t_k}{2}\left\|\by^{k+1} - \by^k\right\|^2 \\
  &=& \frac{t_k}{2}( \left\|\bu^{k+1} - \bu^k\right\|^2 + \left\|\by^{k+1} - \by^k\right\|^2) \\
  &&+ \left\langle W\bu^{k+1} - \by^k, (W^{\top})^{\dagger}(-A^{\top}A - \beta W^{\top}W)(\bu^k - \bu^{k+1}) \right\rangle \\
&=& \frac{t_k}{2}( \left\|\bu^{k+1} - \bu^k\right\|^2 + \left\|\by^{k+1} - \by^k\right\|^2 )\\
&& - \left\langle (W)^{\dagger}(W\bu^{k+1} - \by^k), (A^{\top}A + \beta W^{\top}W)(\bu^k - \bu^{k+1}) \right\rangle \\
&=&\frac{t_k}{2} (\|\bu^{k+1} - \bu^k\|^2 + \|\by^{k+1} - \by^k\|^2 )\\
&&  - \left\langle (W)^{\dagger}W(\bu^{k+1} - \bu^k), (A^{\top}A + \beta W^{\top}W)(\bu^k - \bu^{k+1}) \right\rangle \\
&& - \left\langle (W)^{\dagger}(W\bu^k - \by^k), (A^{\top}A + \beta W^{\top}W)(\bu^k - \bu^{k+1}) \right\rangle \\
&\ge& \frac{t_k}{2}( \left\|\bu^{k+1} - \bu^k\right\|^2 + \left\|\by^{k+1} - \by^k\right\|^2 ) \\
&&- \|W^{\top}(W^{\top})^{\dagger}(A^{\top}A + \beta W^{\top}W)\|_2\|\bu^k - \bu^{k+1}\|^2 \\
&& - \|(W)^{\dagger}(W\bu^k - \by^k)\| \|A^{\top}A + \beta W^{\top}W\|_2 \|\bu^k - \bu^{k+1}\| \\
&\overset{(*)}{\ge}& \frac{t_k}{2}( \left\|\bu^{k+1} - \bu^k\right\|^2 + \left\|\by^{k+1} - \by^k\right\|^2 )\\
&&- \left\|W^{\top}(W^{\top})^{\dagger}(A^{\top}A + \beta W^{\top}W)\right\|_2 \left\|\bu^k - \bu^{k+1}\right\|^2 \\
&& - \frac{\left\|A^{\top} A + (\beta + \rho)W^{\top} W + t_k I\right\|_2  \left\|W^{\dagger}\right\|_2^2}{\rho} \left\|A^{\top}A + \beta W^{\top}W\right\|_2 \left\|\bu^k - \bu^{k+1}\right\|^2 \\
&\ge& \frac{t_k}{2}( \left\|\bu^{k+1} - \bu^k\right\|^2 + \left\|\by^{k+1} - \by^k\right\|^2 ) \\
&&- \left\|W^{\top}(W^{\top})^{\dagger}(A^{\top}A + \beta W^{\top}W)\right\|_2 \left\|\bu^k - \bu^{k+1}\right\|^2 \\
&& - \frac{\left\|A^{\top} A + \beta W^{\top} W + t_k I\right\|_2  \left\|W^{\dagger}\right\|_2^2}{\rho} \left\|A^{\top}A + \beta W^{\top}W\right\|_2 \left\|\bu^k - \bu^{k+1}\right\|^2 \\
&& - \left\|W^{\top} W\right\|_2 \left\|W^{\dagger}\right\|_2^2 \left\|A^{\top}A + \beta W^{\top}W\right\|_2 \left\|\bu^k - \bu^{k+1}\right\|^2 \\
&\overset{(**)}{\ge}& \frac{t_k}{4}\left\|\bu^{k+1} - \bu^k\right\|^2 + \frac{t_k}{2}\left\|\by^{k+1} - \by^k\right\|^2 \\
&\ge& \frac{t_k}{4}( \left\|\bu^{k+1} - \bu^k\right\|^2 + \left\|\by^{k+1} - \by^k\right\|^2 ) \\
&\ge& \frac{c}{4}( \left\|\bu^{k+1} - \bu^k\right\|^2 + \left\|\by^{k+1} - \by^k\right\|^2 ),
   \end{eqnarray*}
where the inequality $(*)$ follows by Lemma \ref{lem:3.4}, the inequality $(**)$ follows from \eqref{equ:tk}.
\end{proof}

Regarding the subdifferential of $h(\bu,\by)$, we readily have the following result.

\begin{lemma}\label{Fder}
Let $h$ be defined in \eqref{hfullrank}. Then for all $(\bu,\by)\in \R^n\times\R^m$ we have
\[
\partial h(\bu,\by)=\left(\nabla_{\bu}h(\bu,\by), \nabla_{\by}(\langle W\bu-\by , \bz(\bu)\rangle+\frac{\rho}{2}\|W\bu-\by\|^2)+\partial(\lambda\|\by\|_0) \right).
\]
\end{lemma}

Next, we give some preliminary results on subgradients of nonsmooth functions and Kurdyka-{\L}ojasiewicz (KL) property. We first recall some subdifferentials (subgradients) for nonsmooth functions in \cite{M18,M06,RW98}.

\begin{definition}\label{defi:A2}
	Consider a function $f: \mathbb{R}^p \rightarrow [-\infty, \infty]$ and a point $\bx \in \operatorname{dom} f$, the regular subdifferential of $f$ at $\bx$ is defined as
	\[
	\widehat{\partial} f(\bx):=\left\{\bv \in \mathbb{R}^p: \liminf _{\bx^{\prime} \rightarrow \bx, \bx^{\prime} \neq \bx} \frac{f\left(\bx^{\prime}\right)-f(\bx)-\left\langle \bv, \bx^{\prime}-\bx\right\rangle}{\left\|\bx^{\prime}-\bx\right\|} \geq 0\right\} ;
	\]
	the subdifferential of the function $f$ at $\bx$ is defined as
	\[
	\partial f(\bx):=\left\{\bv \in \mathbb{R}^p: \exists \bx^k {\underset{f}{\rightarrow}} \bx \text { and } \bv^k \in \widehat{\partial} f\left(\bx^k\right) \text { with } \bv^k \rightarrow \bv \text { as } k \rightarrow \infty\right\};
	\]
	and the horizon subdifferential of the function $f$ at $x$ is defined as
	\[
	\partial^{\infty} f(\bx):=\left\{\bv \in \mathbb{R}^p: \exists \bx^k {\underset{f}{\rightarrow}} \bx, \lambda^k \downarrow 0 \text { and } \bv^k \in \widehat{\partial} f\left(\bx^k\right) \text { s.t. } \lambda^k \bv^k \rightarrow v \text { as } k \rightarrow \infty\right\}.
	\]
\end{definition}

We now recall the KL property for a nonsmooth function \cite{AB10,BS14}.
\begin{definition}\label{app:kli} (KL property)
	Let $h:\Rn\to [-\infty, +\infty]$ be a proper lsc function. Then $h$ is said to have the KL property at $\bar{\ba}\in\dom~\partial h:=\{\ba\in\Rn\; |\; $ $\partial h(\ba)\neq\varnothing\}$ if there exist $\mu\in (0,+\infty]$, a neighborhood  $\cb$ of $\bar{\ba}$ and a function $\xi:[0,\mu)\to \R_+$ such that $\xi$ is concave and continuously differentiable on $(0,\mu)$ and continuous at $0$ with $\xi(0)=0$ and $\xi'(x)>0$ for all
	$x\in(0,\mu)$, and
	the  KL inequality
	\[
	\xi'\big(h(\ba)-h(\bar{\ba})\big) \dist({\bf 0}, \partial h({\ba}))\geq 1
	\]
	holds for all $\ba\in \cb\cap\{\ba\in\Rn\; | \; h(\bar{\ba})<h(\ba)<h(\bar{\ba})+\mu\}$, where $\dist({\bf 0}, \partial h(\bar{\ba})):=\inf\{\|\br\|\; | \; \br\in\partial h(\bar{\ba})\}$.
\end{definition}

If $\xi$ can be chosen as $\xi(s)=c \sqrt{s}$ for some $c>0$, then $h$ is said to have the KL property at $\bar{\ba}$ with an exponent of $1 / 2$. If $h$ has the KL property of exponent $1 / 2$ at each point of dom $\partial h$, then $h$ is called a KL function of exponent $1 / 2$.

The following lemma gives a subgradient lower bound for the successive iterate gap. We refer to Definition \ref{defi:A2} and Lemma \ref{Fder} for the subgradient of $h$.

\begin{lemma}\label{hgribound}
Let $\{(\bu^k,\by^k)\}$ be the sequence generated by Algorithm \ref{algor:rlmem}. For any $k\ge0$, define
\begin{eqnarray*}
\ba^{k+1}&=&\nabla f(\bu^{k+1})+W^{\top}\bz^{k+1}-(A^{\top}A+\beta W^{\top}W)W^{\dagger}(W\bu^{k+1}-\by^{k+1})\\
&&+\beta W^{\top}W\bu^{k+1}+\rho W^{\top}(W\bu^{k+1}-\by^{k+1})
\end{eqnarray*}
and $\bb^{k+1}=-t_k(\by^{k+1}-\by^k)$. Then, $(\ba^{k+1},\bb^{k+1})\in\partial h(\bu^{k+1},\by^{k+1})$ and
\[
\|(\ba^{k+1},\bb^{k+1})\|\le \sqrt{\gamma}\|(\bu^{k+1}-\bu^k,\by^{k+1}-\by^k)\|,
\]
where
\begin{align}
\gamma = &\max\left\{ 
2\left[ 
\left\| W^{\top}(W^{\top})^{\dagger}(-A^{\top}A - \beta W^{\top}W) - (A^{\top}A + \beta W^{\top}W)W^{\dagger}W - t I \right\| 
\right. \right. \nonumber \\
&\left. \left. 
+ \left\| A^{\top}A + \beta W^{\top}W \right\|
\left( 
\left\| A^{\top} A + (\beta + \rho)W^{\top} W + t I \right\| \left\| W^{\dagger} \right\|^2
\right)/\rho 
\right]^2, 
\right. \nonumber \\
&\left. 
2 \left\| (A^{\top}A + \beta W^{\top}W)W^{\dagger} - \rho W^{\top} \right\|^2 + t^2 
\right\} \nonumber
\end{align}
%
%
with $t=\max\{t_k\}$ for any $k$.
\end{lemma}
\begin{proof}
The function $h(\bu,\by)$ under consideration has the following form:
\begin{equation*}
h(\bu,\by)=f(\bu)+\big\langle W\bu-\by , \bz(\bu)\big\rangle+\lambda\|\by\|_0 +\frac{\beta}{2}\|W\bu\|^2+\frac{\rho}{2}\|W\bu-\by\|^2
\end{equation*}
with
\BE\nonumber
\bz(\bu)=(W^{\top})^{\dagger}(-\nabla f(\bu)-\beta W^{\top}W\bu )~\mbox{and}~f(\bu)=\frac{1}{2}\|A\bu-\bu_0\|^2.
\EE
Obviously, it can be easily verified that $\ba^{k+1}=\nabla_{\bu} h(\bu^{k+1},\by^{k+1})$.

By \eqref{rquesy}, there exists an element $\eta^{k+1}\in \lambda \|\by^{k+1}\|_0$ such that
\[
\eta^{k+1}-\bz^{k+1}+t_k(\by^{k+1}-\by^k)+\rho(\by^{k+1}-W\bu^{k+1})={\bf 0},
\]
then
\[
-t_k(\by^{k+1}-\by^k)=\eta^{k+1}-\bz^{k+1}+\rho(\by^{k+1}-W\bu^{k+1}),
\]
it is easy to verify that the right-hand side of the above expression is an element of $\partial_{\by} h(\bu^{k+1},\by^{k+1})$. Therefore, $\bb^{k+1}\in \partial_{\by} h(\bu^{k+1},\by^{k+1})$.
By \eqref{rquesu}, we have
\begin{equation}\label{derirquesu}
\nabla f(\bu^{k+1})+W^{\top}\bz^{k}+\beta W^{\top}W\bu^{k+1}+t_k(\bu^{k+1}-\bu^k)+\rho W^{\top}(W\bu^{k+1}-\by^k)={\bf 0}.
\end{equation}
Substituting \eqref{derirquesu} into the definition of $\ba^{k+1}$ yields
\begin{eqnarray*}
&&\|\ba^{k+1}\|\\
&=&\left\|\nabla f(\bu^{k+1})+W^{\top}(W^{\top})^{\dagger}(-A^{\top}(A\bu^{k+1}-\bu_0)-\beta W^{\top}W\bu^{k+1})\right.\\
&&-(A^{\top}A+\beta W^{\top}W)W^{\dagger}(W\bu^{k+1}-\by^{k+1})+\beta W^{\top}W\bu^{k+1}+\rho W^{\top}(W\bu^{k+1}-\by^{k+1})\\
&&-\nabla f(\bu^{k+1})-W^{\top}(W^{\top})^{\dagger}(-A^{\top}(A\bu^{k}-\bu_0)-\beta W^{\top}W\bu^{k})-\beta W^{\top}W\bu^{k+1}\\
&&\left.-t_k(\bu^{k+1}-\bu^k)-\rho W^{\top}(W\bu^{k+1}-\by^k)\right\|\\
&=&\left\|W^{\top}(W^{\top})^{\dagger}(-A^{\top}A-\beta W^{\top}W)(\bu^{k+1}-\bu^k)+\rho W^{\top}(\by^{k}-\by^{k+1})-t_k(\bu^{k+1}-\bu^k)\right.\\
&&\left.-(A^{\top}A+\beta W^{\top}W)W^{\dagger}(W\bu^{k+1}-\by^{k+1})\right\|\\
&=&\left\|W^{\top}(W^{\top})^{\dagger}(-A^{\top}A-\beta W^{\top}W)(\bu^{k+1}-\bu^k)+\rho W^{\top}(\by^{k}-\by^{k+1})-t_k(\bu^{k+1}-\bu^k)\right.\\
&&-(A^{\top}A+\beta W^{\top}W)W^{\dagger}W(\bu^{k+1}-\bu^{k})-(A^{\top}A+\beta W^{\top}W)W^{\dagger}(W\bu^{k}-\by^{k})\\
&&\left. -(A^{\top}A+\beta W^{\top}W)W^{\dagger}(\by^{k}-\by^{k+1})\right\|\\
&\le &\left\|W^{\top}(W^{\top})^{\dagger}(-A^{\top}A-\beta W^{\top}W)-(A^{\top}A+\beta W^{\top}W)W^{\dagger}W-t_k I\right\|_2\left\|\bu^{k+1}-\bu^{k}\right\|\\
&&+\left\|(A^{\top}A+\beta W^{\top}W)W^{\dagger}-\rho W^{\top}\right\|_2\left\|\by^{k}-\by^{k+1}\right\|\\
&&+\left\|(A^{\top}A+\beta W^{\top}W)\right\|_2\left\|W^{\dagger}(W\bu^{k}-\by^{k})\right\|\\
&\le &\left\|W^{\top}(W^{\top})^{\dagger}(-A^{\top}A-\beta W^{\top}W)-(A^{\top}A+\beta W^{\top}W)W^{\dagger}W-t_k I\right\|_2\left\|\bu^{k+1}-\bu^{k}\right\|\\
&&+\left\|(A^{\top}A+\beta W^{\top}W)W^{\dagger}-\rho W^{\top}\right\|_2\left\|\by^{k}-\by^{k+1}\right\|\\
&&+\left\|(A^{\top}A+\beta W^{\top}W)\right\|_2 \frac{\left\|A^{\top} A+(\beta +\rho)W^{\top} W+t_kI\right\|_2 \left\|W^{\dagger}\right\|_2^2}{\rho}\left\|\bu^{k+1}-\bu^k\right\|,
\end{eqnarray*}
where the last inequality follows from Lemma \ref{lem:3.4}. Hence,
\begin{align}
&\quad \left\|\ba^{k+1}\right\|^2\nonumber\\
 \le& 2 \left\|(A^{\top}A+\beta W^{\top}W)W^{\dagger}-\rho W^{\top}\right\|_2^2\left\|\by^{k}-\by^{k+1}\right\|_2^2 \nonumber\\
&+2\left[\left\|W^{\top}(W^{\top})^{\dagger}(-A^{\top}A-\beta W^{\top}W)(A^{\top}A+\beta W^{\top}W)W^{\dagger}W-t_k I\right\|_2\right.\nonumber\\
&\left.+\left\|(A^{\top}A+\beta W^{\top}W)\right\|_2\left(\left\|A^{\top} A+(\beta +\rho)W^{\top} W+t_kI\right\|_2 \left\|W^{\dagger}\right\|_2^2\right)/\rho\right]^2\|\bu^{k+1}-\bu^k\|^2\nonumber
\end{align}
Additionally, by the definition of $\bb^{k+1}$ we obtain
$\|\bb^{k+1}\|\le t_k\|\by^{k+1}-\by^k\|$.
Therefore,
\[
\|(\ba^{k+1},\bb^{k+1})\|\le \sqrt{\gamma}\|(\bu^{k+1},\by^{k+1})-(\bu^k,\by^k)\|.
\]
\end{proof}

Next, by leveraging the results from \cite{LP18} and \cite{WP21}, we first show that $h$ satisfies the KL property of exponent $1/2$. Subsequently, combining the aforementioned results with this $1/2$-KL property, we establish the linear convergence of the sequence generated by the Algorithm \ref{algor:rlmem}.

Let $\bx=[\bu^{\top},\by^{\top}]^{\top}\in\R^{m+n}$, then
\[
h(\bu,\by)=f(\bu)+\big\langle W\bu-\by ,\bz(\bu)\big\rangle+\lambda\|\by\|_0 +\frac{\beta}{2}\|W\bu\|^2+\frac{\rho}{2}\|W\bu-\by\|^2
\]
can be reformulate as
\begin{equation}\label{symh}
h(\bx)=\bx^{\top}M\bx+\bb^{\top}\bx+\nu+\lambda\sum_{i=n+1}^{m+n}\|x_i\|_0,
\end{equation}
where $M=
\begin{bmatrix}
M_{11} & M_{12} \\
M_{21} & M_{22}
\end{bmatrix}
$, $\bb=[\bb_1^{\top},\bb_2^{\top}]^{\top}$ with
\BE\label{equ:group}
\begin{aligned}
&M_{11}=\frac12\big(A^{\top}A+(\rho-\beta) W^{\top}W-W^{\top}(W^{\top})^{\dagger}A^{\top}A-A^{\top}AW^{\dagger}W\big),\\ &W_{12}=\frac12(\beta W^{\top}W+A^{\top}A)W^{\dagger}-\rho W^{\top}),\quad M_{21}=M_{12}^{\top},\quad M_{22}=\rho \frac12 I,\\ &\bb_1=-A^{\top}\bu_0+W^{\top}(W^{\top})^{\dagger}A^{\top}\bu_0,\\ &\bb_2=-(W^{\top})^{\dagger}A^{\top}\bu_0,\quad \nu=\frac12\bu_0^{\top}{\bu_0}.
\end{aligned}
\EE

\begin{theorem}\label{hKL}(Properties of function $h$)
The function $h$ has the $K L$ property of exponent $1 / 2$ at all critical points.
\end{theorem}

\begin{proof}
Let $g(\bx)=\bx^{\top}M\bx+\bb^{\top}\bx+\nu$, where $M$, $\bb$ and $\nu$ are defined in \eqref{equ:group}. Fix an arbitrary $\bar{\bx} \in \operatorname{crit} h$, where $\operatorname{crit} h$ denotes the set of critical points of $h(\bu,\by)$. Let $\cj=\{i| \bar{x}_i\neq 0, i=n+1,\ldots,n+m\}$ and $\ci=[n]$.  We define $g_{\ci\cup\cj}(\bz):=g(I_{\ci\cup\cj} \bz)$ for $\bz\in\R^{|\ci\cup\cj|}$, where $I_{\ci\cup\cj}\in\R^{(n+m)\times |\ci\cup\cj|}$ is formed by extracting columns of the identity matrix using the elements of $\ci\cup\cj$ as indices. By \cite[Corollary 5.2]{LP18}, we know $g_{\ci\cup\cj}$ is a KL function of exponent $1 / 2$. Then, there exist $\delta_1>0, \eta_1>0$ and $c_1>0$ such that for all $\bz \in \mathbb{B}\left(\bar{\bx}_{\ci\cup\cj}, \delta_1\right) \cap\{\bz |g_{\ci\cup\cj}\left(\bar{\bx}_{\ci\cup\cj}\right)<g_{\ci\cup\cj}(\bz)<g_{\ci\cup\cj}\left(\bar{\bx}_{\ci\cup\cj}\right)+\eta_1\}$,
\begin{eqnarray}\label{fuzhuKLinequ1}
\operatorname{dist}\left(0, \partial_{g_{\ci\cup\cj}}(\bz)\right) \geq c_1 \sqrt{g_{\ci\cup\cj}(\bz)-g_{\ci\cup\cj}\left(\bar{\bx}_{\ci\cup\cj}\right)} .
\end{eqnarray}
Take $\eta_2 \in( 0, \lambda/ 3]$. Since $g(\bx)$ is continues, there exists $\delta_2>0$ such that
\begin{eqnarray}\label{fuzhuKLinequ2}
|g(\bx)-g(\bar{\bx})|<\eta_2 \quad \forall \bx \in \mathbb{B}\left(\bar{\bx}, \delta_2\right) \cap \operatorname{dom} \partial g.
\end{eqnarray}
Take $\delta=\min \left(\delta_1, \delta_2\right)$ and $\eta=\min \left(\eta_1, \eta_2\right)$. Pick an arbitrary $\bx$ from the set $\mathbb{B}(\bar{\bx}, \delta) \cap \{\bx |h(\bar{\bx})<h(\bx)<h(\bar{\bx})+\eta\}$. We prove the arguments by two cases.
\begin{itemize}
\item Case 1: $\bx \in \operatorname{dom} \partial h$. We have
\end{itemize}
\begin{eqnarray}\label{fuzhuKLinequ3}
\partial h(\bx) \subseteq \nabla g(\bx)+\lambda\sum_{i=n+1}^{m+n}\partial\|x_i\|_0.
\end{eqnarray}
Next, we prove $g(\bx)>g(\bar{\bx})$ by contradiction. 
Suppose $g(\bx) \leq g(\bar{\bx})$. Combining with \eqref{fuzhuKLinequ2}, we have $g(\bx)<g(\bar{\bx})+\eta_2$.
 Moreover, since $h(\bar{\bx})<h(\bx)<h(\bar{\bx})+\eta$, that is,
 \begin{equation}\label{contra1}
g(\bar{\bx})+ \lambda\sum_{i=n+1}^{m+n}\|\bar{x}_i\|_0 <g(\bx)+\lambda\sum_{i=n+1}^{m+n}\|x_i\|_0<g(\bar{\bx})+ \lambda\sum_{i=n+1}^{m+n}\|\bar{x}_i\|_0+\eta.
 \end{equation}
By  \eqref{contra1}, $g(\bx)\le g(\bar{\bx})$ and $g(\bx)<g(\bar{\bx})+\eta_2$,  we have
\[\sum_{i=n+1}^{m+n}\|\bar{x}_i\|_0 \leq\sum_{i=n+1}^{m+n}\|x_i\|_0-1<\sum_{i=n+1}^{m+n}\|\bar{x}_i\|_0+\frac{1}{\lambda}\left(\eta+\eta_2\right)-1<\sum_{i=n+1}^{m+n}\|\bar{x}_i\|_0.
\]
Therefore, $g(\bx)>g(\bar{\bx})$.

By $h(\bx)<h(\bar{\bx})+\eta$, we obtain that $\sum_{i=n+1}^{m+n}\|x_i\|_0 <\sum_{i=n+1}^{m+n}\|\bar{x}_i\|_0+\eta/\lambda$, that is, $\sum_{i=n+1}^{m+n}\|x_i\|_0 \leq\sum_{i=n+1}^{m+n}\|\bar{x}_i\|_0$. In addition, according to $\bx\in\mathbb{B}(\bar{\bx},\delta)$, when $\delta$ is sufficiently small, we can deduce that $\sum_{i=n+1}^{m+n}\|x_i\|_0 \geq\sum_{i=n+1}^{m+n}\|\bar{x}_i\|_0$ and $\{i|x_i\neq 0, i=n+1,\ldots, m+n\}\supseteq \{i|\bar{x}_i\neq 0, i=n+1,\ldots, m+n\}$. Thus, $\sum_{i=n+1}^{m+n}\|x_i\|_0=\sum_{i=n+1}^{m+n}\|\bar{x}_i\|_0$ and
\BE\label{fuzhuKLinequ4}
\{i|x_i\neq 0, i=n+1,\ldots, m+n\}=\{i|\bar{x}_i\neq 0, i=n+1,\ldots, m+n\}=\cj.
\EE
We denote $\zeta^*=\nabla g(x)$. By virtue of \eqref{fuzhuKLinequ3}, we obtain
\begin{eqnarray}\label{fuzhuKLinequ5}
\operatorname{dist}(0, \partial h(\bx)) & \geq& \operatorname{dist}(0, \nabla g(\bx)+\lambda\sum_{i=n+1}^{m+n}\partial\|x_i\|_0) \nonumber\\
& =&\min_{\zeta = \nabla g(\bx), \xi \in \lambda \sum_{i=n+1}^{m+n}\partial\|x_i\|_0}\|\zeta+\xi\|=\|\zeta_{\ci\cup\cj}^*\| .
\end{eqnarray}
By the definition of $g_{\ci\cup\cj}$, we have
\begin{eqnarray}\label{fuzhuKLinequ6}
\nabla g_{\ci\cup\cj}\left(\bx_{\ci\cup\cj}\right) =I_{\ci\cup\cj}^{\top} \nabla g\left(I_{\ci\cup\cj} \bx_{\ci\cup\cj}\right).
\end{eqnarray}
In addition, using $\zeta^* = \nabla g(\bx)$ and $\bx=I_{\ci\cup\cj} \bx_{\ci\cup\cj}$, we have
\[\zeta_{\ci\cup\cj}^* = I_{\ci\cup\cj}^{\top} \nabla g(x)=I_{\ci\cup\cj}^{\top} \nabla g\left(I_{\ci\cup\cj}\bx_{\ci\cup\cj}\right).
\]
From the equation \eqref{fuzhuKLinequ6}, it follows that $\zeta_{\ci\cup\cj}^* = \nabla g_{\ci\cup\cj}\left(\bx_{\ci\cup\cj}\right)$. Moreover, by \eqref{fuzhuKLinequ5}, we obtain
\begin{eqnarray}\label{fuzhuKLinequ7}
\operatorname{dist}(0, \partial h(\bx)) \geq\left\|\zeta_{\ci\cup\cj}^*\right\| = \operatorname{dist}\left(0, \nabla g_{\ci\cup\cj}\left(\bx_{\ci\cup\cj}\right)\right) .
\end{eqnarray}
Recall that $\bx \in\{\bv |h(\bar{\bx})<h(\bv)<h(\bar{\bx})+\eta\}$. By invoking \eqref{fuzhuKLinequ4}, it follows that
\[
h(\bx)=g\left(I_{\ci\cup\cj} \bx_{\ci\cup\cj}\right)+\lambda|\cj|=g_{\ci\cup\cj}\left(\bx_{\ci\cup\cj}\right) +\lambda|\cj| \text { and } h(\bar{\bx})=g_{
\ci\cup\cj}\left(\bar{\bx}_{\ci\cup\cj}\right)+\lambda|\cj| .
\]
Thus, $\bx_{\ci\cup\cj} \in\{\bz|g_{\ci\cup\cj}\left(\bar{\bx}_{\ci\cup\cj}\right)<g_{\ci\cup\cj}(\bz)<g_{\ci\cup\cj}\left(\bar{\bx}_{\ci\cup\cj}\right)+\eta_1\}$. By $\bx \in \mathbb{B}\left(\bar{\bx}, \delta\right)$, $\delta\le \delta_1$ and $\|\bx-\bar{\bx}\|=\|\bx_{\ci\cup\cj}-\bar{\bx}_{\ci\cup\cj}\|$, we have $\bx_{\ci\cup\cj} \in \mathbb{B}\left(\bar{\bx}_{\ci\cup\cj}, \delta_1\right)$, by \eqref{fuzhuKLinequ7} and \eqref{fuzhuKLinequ1},
\begin{eqnarray*}
\operatorname{dist}(0, \partial h(\bx)) &\geq& \operatorname{dist}\left(0, \partial g_{\ci\cup\cj}\left(\bx_{\ci\cup\cj}\right)\right) \\
&\geq& c_1 \sqrt{g_{\ci\cup\cj}\left(\bx_{\ci\cup\cj}\right)-g_{\ci\cup\cj}\left(\bar{\bx}_{\ci\cup\cj}\right)}=c_1 \sqrt{h(\bx)-h(\bar{\bx})} .
\end{eqnarray*}
\begin{itemize}
  \item Case 2: $\bx \notin \operatorname{dom} \partial h$.
\end{itemize}
In this case, $\partial h(\bx)=\varnothing$, and $\operatorname{dist}(0, \partial h (\bx))=\infty$. This implies that the final inequality holds automatically.

Now, by the arbitrariness of $\bx$ in $\mathbb{B}(\bar{\bx}, \delta) \cap\{\bx|h(\bar{\bx})<h(\bx)<h(\bar{\bx})+\eta]$, the final inequality implies that $h$ has the KL property of exponent $1 / 2$ at $\bar{\bx}$. Given the arbitrariness of $\bar{\bx} \in \operatorname{crit} h$, the desired conclusion follows.
\end{proof}

In the following, utilizing the properties of $h$ and the above results, we prove that the sequence generated by the Algorithm \ref{algor:rlmem} converges linearly.

\begin{theorem}\label{sequcon}
Under the same assumptions as in Lemma \ref{hfunbound},  let $\{(\bu^k,\by^k)\}$ be the sequence generated by Algorithm \ref{algor:rlmem} which is assumed to be bounded. 
Then the sequence $\{(\bu^k,\by^k)\}$ converges at least linearly to a limiting critical point $(\bu^*,\by^*)$ of $h$, i.e., there exist constants $c_1>0$ and $\tau\in [0,1)$ such that
\[
\|(\bu^k,\by^k)-(\bu^*,\by^*)\|\le c_1\tau^k~~\mbox{for}~k=0,1,\ldots
\]
\end{theorem}

\begin{proof}
Let $(\bu^*,\by^*)$ be an accumulation point of the sequence $\{\bu^k,\by^k\}$. Then there exists a subsequence  $\{\bu^{k_t},\by^{k_t}\}$ converging to  $(\bu^*,\by^*)$. We note that $\lambda \|\by\|_0$ is lower semicontinuous, thus
\BE\label{eq:phi-inf}
\liminf_{t\to\infty}\lambda\|\by^{k_t}\|_0\ge\lambda\|\by^{*}\|_0.
\EE
Using \eqref{rquesy} we obtain
\begin{eqnarray}\label{eq:ykystar}
\lambda \|\by^{k+1}\|_0&+&\big\langle W\bu^{k+1}-\by^{k+1} ,\bz^{k+1}) \big\rangle +\frac{t_k}{2}\|\by^{k+1}-\by^k\|^2+\frac{\rho}{2}\|W\bu^{k+1}-\by^{k+1}\|^2 \nonumber\\
&\le&  \lambda \|\by^*\|_0+\big\langle W\bu^{k+1}-\by^* ,\bz^{k+1}) \big\rangle +\frac{t_k}{2}\|\by^*-\by^k\|^2+\frac{\rho}{2}\|W\bu^{k+1}-\by^*\|^2.
\end{eqnarray}
By virtue of Lemma \ref{hfunbound}, we have
\begin{eqnarray*}
h(\bu^0,\by^0)-h(\bu^{k+1},\by^{k+1})&=&\sum_{i=0}^{k}(h(\bu^i,\by^i)-h(\bu^{i+1},\by^{i+1}))\\
&\ge&\frac c4\sum_{i=0}^{k}(\|\bu^{i+1}-\bu^i\|^2+\|\by^{i+1}-\by^i\|^2).
\end{eqnarray*}
By the boundedness of the sequence $\{(\bu^k, \by^k)\}$, $\{h(\bu^k, \by^k)\}$ is bounded below, and we have
\[
\lim\limits_{k\to\infty}\|(\bu^{k+1},\by^{k+1})-(\bu^k,\by^k)\|=0.
\]
Hence, taking $k=k_t-1$ in \eqref{eq:ykystar} and let $t\to\infty$, we obtain
\[
\limsup_{t\to\infty}\lambda\|\by^{k_t}\|_0 \le \lambda\|\by^*\|_0.
\]
This, together with \eqref{eq:phi-inf}, implies that
\[
\lim_{t\to\infty}\lambda \|\by^{k_t}\|_0 =\lambda \|\by^*\|_0.
\]
Therefore,
\begin{eqnarray}\label{eq:F-lim}
\lim_{t\to\infty}h(\bu^{k_t},\by^{k_t}) &=& \lim_{t\to\infty}( f(\bu^{k_t})+\langle W\bu^{k_t}-\by^{k_t} , \bz(\bu^{k_t},\by^{k_t})\rangle+\lambda\|\by^{k_t}\|_0 +\frac{\beta}{2}\|W\bu^{k_t}\|^2\nonumber\\
&&+\frac{\rho}{2}\|W\bu^{k_t}-\by^{k_t}\|^2 ) \nonumber\\
&=&h(\bu^*,\by^*).
\end{eqnarray}

According to \cite[Theorem 2.9]{AB13}, if $h$ satisfies the KL property, and the sequence $\{(\bu^k, \by^k)\}$ simultaneously satisfies the sufficient decrease condition, the relative error condition, and the continuity condition, then the sequence converges to a critical point of $h$. Clearly, Theorem \ref{hKL} shows that  $h$ satisfies the KL property, while Lemma \ref{hfunbound}, Lemma \ref{hgribound} and \eqref{eq:F-lim} demonstrate that the sequence $\{\bu^k, \by^k\}$ satisfies the sufficient decrease condition, the relative error condition, and the continuity condition, respectively. Thus, we can conclude that the sequence $\{(\bu^k,\by^k)\}$ converges to a limiting critical point $\{(\bu^*,\by^*)\}$ of the function $h$. Moreover, by \cite[Theorem 3.4]{AB10}, there exist constants $c_1>0$ and $\tau \in[0,1)$ such that
\[
\|(\bu^{k+1}, \by^{k+1})-(\bu^*, \by^*)\|_F \leq c_1 \tau^k \quad \text { for } k=0,1, \ldots
\]
Therefore, the sequence $\{(\bu^k,\by^k)\}$ converges at least linearly to the limiting critical point $\{(\bu^k,\by^k)\}$ of function $h$.
\end{proof}

Next, we demonstrate that the limit point of the sequence generated by the Algorithm \ref{algor:rlmem} is a KKT point of model \eqref{equ:reformu}.

\begin{theorem}\label{finalres1}
Under the same assumptions as in Lemma \ref{hfunbound}, let $ \{(\bu^k, \by^k)\}$ be the bounded sequence generated by Algorithm \ref{algor:rlmem}. Then, every limit point $(\bu^*, \by^*)$ of $\{(\bu^k, \by^k)\}$ is a KKT point of \eqref{equ:reformu}.
\end{theorem}

\begin{proof}
We firstly write the KKT point of \eqref{equ:reformu}. A point $(\bu, \by)$ is called a KKT point of  \eqref{equ:reformu} if there exists $\bz$ such that
\[
{\bf 0}=\nabla f(\bu)+W^{\top}\bz+\beta W^{\top}W\bu, \quad {\bf 0}\in-\bz +\partial \lambda \|\by\|_0, \quad W\bu-\by={\bf 0}.
\]
By \eqref{rquesu} we have
\begin{eqnarray*}
&&\nabla f(\bu^{k+1})+W^{\top}\bz^k+\beta W^{\top}W\bu^{k+1}+t_k(\bu^{k+1}-\bu^k)+\rho W^{\top}(W\bu^{k+1}-\by^k)={\bf 0}.
\end{eqnarray*}
Let $k\to\infty$,  by the continuity of $\nabla f(\bu)$ and $\bz(\bu)$ together with Theorem \ref{sequcon}, we have
\begin{eqnarray}\label{fuzhulimkkt}
\nabla f(\bu^*)+W^{\top}\bz^*+\beta W^{\top}W\bu^{*}+\rho W^{\top}(W\bu^{*}-\by^*)={\bf 0}.
\end{eqnarray}
Then, by $\nabla f(\bu^*)+\beta W^{\top}W\bu^{*}=(I-W^{\top}(W^{\top})^{\dagger})(\nabla f(\bu^*)+\beta W^{\top}W\bu^{*})+W^{\top}(W^{\top})^{\dagger}$ $(\nabla f(\bu^*)+\beta W^{\top}W\bu^{*})$, we have
\begin{eqnarray*}
&& (I-W^{\top}(W^{\top})^{\dagger})(\nabla f(\bu^*)+\beta W^{\top}W\bu^{*})+W^{\top}(W^{\top})^{\dagger}(\nabla f(\bu^*)+\beta W^{\top}W\bu^{*})\\
&&+W^{\top}\bz^*+\rho W^{\top}(W\bu^{*}-\by^*)={\bf 0}.
\end{eqnarray*}
Obviously, $(I-W^{\top}(W^{\top})^{\dagger})(\nabla f(\bu^*)+\beta W^{\top}W\bu^{*})\in\Range(I-W^{\top}(W^{\top})^{\dagger})$ and $W^{\top}(W^{\top})^{\dagger}(\nabla f(\bu^*)+\beta W^{\top}W\bu^{*})
+W^{\top}\bz^*+\rho W^{\top}(W\bu^{*}-\by^*)\in\Range(W^{\top})$. Furthermore, since $\Range(I-W^{\top}(W^{\top})^{\dagger}) $ is the orthogonal complement of $\Range(W^{\top})$, it follows that $\Range(I-W^{\top}(W^{\top})^{\dagger}) \cap\Range(W^{\top})=\{\mathbf{0}\}$. Thus, we have
\[
(I-W^{\top}(W^{\top})^{\dagger})(\nabla f(\bu^*)+\beta W^{\top}W\bu^{*})={\bf 0}
\]
and 
\[
W^{\top}(W^{\top})^{\dagger}(\nabla f(\bu^*)+\beta W^{\top}W\bu^{*})
+W^{\top}\bz^*+\rho W^{\top}(W\bu^{*}-\by^*)={\bf 0}.
\]

By the definition of $\bz^*$, we have $W^{\top}(W^{\top})^{\dagger}(\nabla f(\bu^*)+\beta W^{\top}W\bu^{*})
+W^{\top}\bz^*={\bf 0}$. Then,
$\rho W^{\top}(W\bu^{*}-\by^*)={\bf 0}$.  Since $W^{\top}$ is of full column rank, then 
\BE\label{equ:kkt1}
W\bu^{*}-\by^*={\bf 0}.
\EE
Substituting the result into \eqref{fuzhulimkkt}, we have
\BE\label{equ:kkt2}
\nabla f(\bu^*)+W^{\top}\bz^*+\beta W^{\top}W\bu^{*}=0.
\EE

Next, we prove the second equation of the KKT conditions. Since $(\bu^*,\by^*)$ is the limit point of sequence of $\{\bu^k,\by^k\}$, let $k\to\infty$, Combining with \eqref{rquesy} and $(W\bu^{*}-\by^*)={\bf 0}$, we have
\begin{eqnarray*}
(W^{\top})^{\dagger}(A^{\top}(A\bu^*-\bu_0)-\beta W^{\top}W\bu^*)+\eta^*={\bf 0},
\end{eqnarray*}
where $\eta^*\in \partial \lambda \|\by^*\|_0$. That is, 
\BE\label{equ:kkt3}
-\bz^*+\eta^*={\bf 0}.
\EE
Combining equations \eqref{equ:kkt1}--\eqref{equ:kkt3}, this shows that $(\bu^*,\by^*)$ is KKT point of \eqref{equ:reformu}.
\end{proof}

\begin{remark}
From the proof process of Theorem \ref{finalres1}, it is easy to verify that if $(\bu^*,\by^*)$ is a KKT point of \eqref{equ:reformu}, then $(\bu^*,\by^*)$ must be a critical point of $h$. Conversely, if $(\bu^*,\by^*)$ is a critical point of $h$ and satisfies $W\bu^* = \by^*$, then $(\bu^*,\by^*)$ is also a KKT point of \eqref{equ:reformu}.
\end{remark}

\subsection{Convergence Analysis for Algorithm \ref{algor:clmem}}
Next, we conduct the convergence analysis of Algorithm \ref{algor:clmem}. Similar to the preceding subsection, we first demonstrate that the constraint violation at the current iterate can be bounded by the norm of difference between the updated value and the current iterate.

\begin{lemma}\label{lemmaguk+1uk}
Let $\{(\bu^k,\by^k)\}$ be the sequence generated by Algorithm \ref{algor:clmem}, if $\rank((W^{\top})_{\mathcal{C}^k})>0$, then
\[
\left\|\bu^{k+1}-\bu^k\right\|\ge \frac{\rho\left\|(((W^{\top})_{\mathcal{C}^k})^{\dagger})^{\top} (W\bu^k-\by^k)_{\mathcal{C}^k}\right\|}{\left\|A^{\top} A+(\beta +\rho)W^{\top} W+t_kI\right\|\left\|((W^{\top})_{\mathcal{C}^k})^{\dagger}\right\|^2} .
\]
\end{lemma}
\begin{proof}
Let $(W^{\top})_{\mathcal{C}^k}=U_r\Sigma_rV_r^{\top}$ with $\rank((W^{\top})_{\mathcal{C}^k})=r>0$.  Then
\begin{eqnarray}
\left\|(((W^{\top})_{\mathcal{C}^k})^{\dagger})^{\top} (W\bu^k-\by^k)_{\mathcal{C}^k}\right\| &=&\left\|U_r\Sigma_r^{-1}V_r^{\top}(W\bu^k-\by^k)_{\mathcal{C}^k}\right\| 
\nonumber\\
&\le& \left\|\Sigma_r^{-1}\right\|_2\left \|V_r^{\top}(W\bu^k-\by^k)_{\mathcal{C}^k}\right\|\nonumber\\
&=&\left\|((W^{\top})_{\mathcal{C}^k})^{\dagger}\right\|_2 \left\|V_rV_r^{\top}(W\bu^k-\by^k)_{\mathcal{C}^k}\right\|\nonumber\\
&=&\left\|((W^{\top})_{\mathcal{C}^k})^{\dagger}\right\|_2 \left\|((W^{\top})_{\mathcal{C}^k})^{\dagger}(W^{\top})_{\mathcal{C}^k} (W\bu^k-\by^k)_{\mathcal{C}^k}\right\|\nonumber\\
&=&\left\|((W^{\top})_{\mathcal{C}^k})^{\dagger}\right\|_2 \left\|((W^{\top})_{\mathcal{C}^k})^{\dagger}W^{\top}(W\bu^k-\by^k)\right\|\label{gwwandw}.
\end{eqnarray}
where the last equality follow from \eqref{equ:yk}, i.e., for $i\in\ci^k$, $\by_i^{k}=(W\bu^{k})_i$. By Theorem \ref{solutionguy}, we have
\begin{eqnarray*}
&& \left\|\bu^{k+1}-\bu^k\right\|\\
&=&\left\|(A^{\top} A+(\beta+\rho) W^{\top}W+t_k I)^{-1}(A^{\top} \bu_0-W^{\top} \bz^k+\rho W^{\top} \by^k+t_k\bu^k)-\bu^k\right\| \\
&=&\left\|(A^{\top} A+(\beta+\rho) W^{\top}W+t_k I)^{-1}\left(A^{\top} \bu_0-W^{\top} \bz^k+\rho W^{\top} \by^k+t_k\bu^k-A^{\top}A \bu^k\right.\right.\\
&&\left.\left.-(\beta+\rho) W^{\top} W\bu^k-t_k\bu^k\right)\right\| \\
& =&\left\|(A^{\top} A+(\beta+\rho) W^{\top}W+t_k I)^{-1}\left(A^{\top} (\bu_0-A\bu^k)-W^{\top} \bz^k+\rho W^{\top} \by^k\right.\right.\\
&&\left.\left.-(\beta+\rho) W^{\top} W\bu^k\right)\right\|\\
& \geq& \frac{ \left\|A^{\top} (\bu_0-A\bu^k)-W^{\top} \bz^k+\rho W^{\top} \by^k-(\beta+\rho) W^{\top} W\bu^k\right\|}{\left\|A^{\top} A+(\beta +\rho)W^{\top} W+t_kI\right\|_2}\\
&=&\frac{1}{\left\|A^{\top} A+(\beta +\rho)W^{\top} W+t_kI\right\|_2}\left\| A^{\top} (\bu_0-A\bu^k)-\beta W^{\top}W\bu^k -\rho W^{\top}(W\bu^k-\by^k)\right.\\
&&\left.-(W^{\top})_{\mathcal{C}^k}((W^{\top})_{\mathcal{C}^k})^{\dagger} (-A^{\top}(A\bu^k-\bu_0)-\beta W^{\top}W\bu^k)\right\|\\
&=&\frac{1}{\left\|A^{\top} A+(\beta +\rho)W^{\top} W+t_kI\right\|_2}\left\| -\rho W^{\top}(W\bu^k-\by^k)\right.\\
&&\left.+[I-(W^{\top})_{\mathcal{C}^k}((W^{\top})_{\mathcal{C}^k})^{\dagger}] \big[-A^{\top}(A\bu^k-\bu_0)-\beta W^{\top}W\bu^k\big]\right\| \\
&\ge&\frac{1}{\left\|A^{\top} A+(\beta +\rho)W^{\top} W+t_kI\right\|_2\left\|((W^{\top})_{\mathcal{C}^k})^{\dagger}\right\|_2} \left\|((W^{\top})_{\mathcal{C}^k})^{\dagger}[I-(W^{\top})_{\mathcal{C}^k} ((W^{\top})_{\mathcal{C}^k})^{\dagger}]\right.\\
&&\left.\big[-A^{\top}(A\bu^k-\bu_0)-\beta W^{\top}W\bu^k\big]-\rho ((W^{\top})_{\mathcal{C}^k})^{\dagger}W^{\top}(  W\bu^k-\by^k)\right\|\\
&=& \frac{\rho}{\left\|A^{\top} A+(\beta +\rho)W^{\top} W+t_kI\right\|_2\left\|((W^{\top})_{\mathcal{C}^k})^{\dagger}\right\|_2}\left\| ((W^{\top})_{\mathcal{C}^k})^{\dagger}W^{\top}(  W\bu^k-\by^k)\right\|\\
&\ge& \frac{\rho}{\left\|A^{\top} A+(\beta +\rho)W^{\top} W+t_kI\right\|_2\left\|((W^{\top})_{\mathcal{C}^k})^{\dagger}\right\|_2^2} \left\|(((W^{\top})_{\mathcal{C}^k})^{\dagger})^{\top} (W\bu^k-\by^k)_{\mathcal{C}^k}\right\|
\end{eqnarray*}
where the last inequality follows from \eqref{gwwandw}.
\end{proof}

In what follows, we present the result that the sequence of function values corresponding to the sequence generated by Algorithm \ref{algor:clmem} is monotonically decreasing.

\begin{lemma}\label{gfunbound}
Let $\{(\bu^k,\by^k)\}$ be the sequence generated by Algorithm \ref{algor:clmem}. Let $\rho\ge \max\{t, 1\}$ and $t=\max_{k}\{t_k\}$,
\begin{equation}\label{equ:tk2}
\begin{aligned}
t_k\ge \max &\left\{ \max_{k}\left\{4 \left[\left\|((W^{\top})_{\mathcal{C}^k})^{\dagger}\right\| \left\|W\right\|\left\|A^{\top}A+\beta W^{\top}W\right\|\right. \right.\right. \\
&+\left\|W^{\top} W\right\| \left\|((W^{\top})_{\mathcal{C}^k})^{\dagger}\right\|^2\left\|A^{\top}A+\beta W^{\top}W\right\|  \\
&\left.\left.\left.+\left(\left\|A^{\top}A+\beta W^{\top}W\right\|+1\right)\left\|((W^{\top})_{\mathcal{C}^k})^{\dagger}\right\|^2 \left\|A^{\top}A+\beta W^{\top}W\right\|\right]\right\},~ c \right\},
\end{aligned}
\end{equation}
$c>0$, then
\[
h(\bu^k,\by^k)-h(\bu^{k+1},\by^{k+1})\ge \frac{c}{4}\left\|\bu^{k+1}-\bu^k\right\|^2.
\]
\end{lemma}

\begin{proof}
By \eqref{cquesu} and \eqref{cquesy} we have
\begin{eqnarray}
f(\bu^k)+\frac{\beta}{2}\left\|W\bu^k\right\|^2&&+\left\langle W\bu^k-\by^k,\bz^k\right\rangle+\frac{\rho}{2}\left\|W\bu^k-\by^k\right\|^2\ge  f(\bu^{k+1})+\frac{\beta}{2}\left\|W\bu^{k+1}\right\|^2\nonumber\\
&&+\left\langle W\bu^{k+1}-\by^k,\bz^{k}\right\rangle+\frac{\rho}{2} \left\|W\bu^{k+1}-\by^k\right\|^2 +\frac{t_k}{2}\left\|\bu^{k+1}-\bu^k\right\|^2\label{inequgu}
\end{eqnarray}
and
\begin{eqnarray}
h(\bu^{k+1},\by^k)\ge h(\bu^{k+1},\by^{k+1})\label{inequgy}.
\end{eqnarray}
(i) If $\rank((W^{\top})_{\mathcal{C}^k})>0$, using \eqref{inequgu} and \eqref{inequgy} we have
\begin{eqnarray*}
&& h(\bu^k,\by^k)-h(\bu^{k+1},\by^{k+1})\\
&=& h(\bu^k,\by^k)-h(\bu^{k+1},\by^k)+h(\bu^{k+1},\by^k)-h(\bu^{k+1},\by^{k+1})\\
&=& f(\bu^k)+\lambda\left\|\by^k\right\|_0+\frac{\beta}{2}\left\|W\bu^k\right\|^2 +\left\langle W\bu^k-\by^k,\bz(\bu^{k},\by^k)\right\rangle +\frac{\rho}{2}\left\|W\bu^k-\by^k\right\|^2\\
&&- f(\bu^{k+1})-\lambda\left\|\by^k\right\|_0 -\frac{\beta}{2}\left\|W\bu^{k+1}\right\|^2-\left\langle W\bu^{k+1}-\by^k,\bz(\bu^{k+1},\by^k)\right\rangle \\
&&-\frac{\rho}{2}\left\|W\bu^{k+1}-\by^k\right\|^2+h(\bu^{k+1},\by^k)- h(\bu^{k+1},\by^{k+1})\\
&\ge&\frac{t_k}{2}\left\|\bu^{k+1}-\bu^k\right\|^2+\left\langle W\bu^{k+1}-\by^k,\bz(\bu^k,\by^k)\right\rangle- \left\langle W\bu^{k+1}-\by^k,\bz(\bu^{k+1},\by^k)\right\rangle\\
&=&\frac{t_k}{2}\left\|\bu^{k+1}-\bu^k\right\|^2+\left\langle W\bu^{k+1}-\by^k, \bz(\bu^k,\by^k)-\bz(\bu^{k+1},\by^k)\right\rangle\\
&=&\frac{t_k}{2}\left\|\bu^{k+1}-\bu^k\right\|^2-\left\langle (W\bu^{k+1}-\by^k)_{\mathcal{C}^k},((W^{\top})_{\mathcal{C}^k})^{\dagger}(A^{\top}A+\beta W^{\top}W)(\bu^k-\bu^{k+1})\right\rangle\\
&=&\frac{t_k}{2}\left\|\bu^{k+1}-\bu^k\right\|^2 \\
&&-\left\langle(((W^{\top})_{\mathcal{C}^k})^{\dagger})^{\top} (W\bu^{k+1}-\by^k)_{\mathcal{C}^k},(A^{\top}A+\beta W^{\top}W)(\bu^k-\bu^{k+1})\right\rangle\\
&=&\frac{t_k}{2}\left\|\bu^{k+1}-\bu^k\right\|^2 \\
&&-\left\langle(((W^{\top})_{\mathcal{C}^k})^{\dagger})^{\top} (W\bu^{k+1}-W\bu^k)_{\mathcal{C}^k},(A^{\top}A+\beta W^{\top}W)(\bu^k-\bu^{k+1})\right\rangle\\
&&-\left\langle (((W^{\top})_{\mathcal{C}^k})^{\dagger})^{\top} (W\bu^{k}-\by^k)_{\mathcal{C}^k},(A^{\top}A+\beta W^{\top}W)(\bu^k-\bu^{k+1})\right\rangle\\
& \ge&\frac{t_k}{2}\left\|\bu^{k+1}-\bu^k\right\|^2 -\left\|((W^{\top})_{\mathcal{C}^k})^{\dagger}\right\|_2\left\|W\right\|_2 \left\|A^{\top}A+\beta W^{\top}W\right\|_2\left\|\bu^k-\bu^{k+1}\right\|_2^2\\
&&-\left\| (((W^{\top})_{\mathcal{C}^k})^{\dagger})^{\top} (W\bu^{k}-\by^k)_{\mathcal{C}^k}\right\|_2\left\|A^{\top}A+\beta W^{\top}W\right\|_2\left\|\bu^k-\bu^{k+1}\right\|_2\\
& \overset{(*)}{\ge}&\frac{t_k}{2}\left\|\bu^{k+1}-\bu^k\right\|^2 -\left\|((W^{\top})_{\mathcal{C}^k})^{\dagger}\right\|_2\left\|W\right\|_2 \left\|A^{\top}A+\beta W^{\top}W\right\|_2\left\|\bu^k-\bu^{k+1}\right\|_2^2\\
&& -\frac{\left\|A^{\top} A+(\beta +\rho)W^{\top} W+t_kI\right\|_2 \left\|((W^{\top})_{\mathcal{C}^k})^{\dagger}\right\|_2^2}{\rho} \left\|A^{\top}A+\beta W^{\top}W\right\|_2 \left\|\bu^k-\bu^{k+1}\right\|^2\\%
& \ge&\frac{t_k}{2}\left\|\bu^{k+1}-\bu^k\right\|^2- \left\|((W^{\top})_{\mathcal{C}^k})^{\dagger}\right\|_2\left\|W\right\|_2 \left\|A^{\top}A+\beta W^{\top}W\right\|_2\left\|\bu^k-\bu^{k+1}\right\|^2\\
&& -\frac{\left\|A^{\top} A+\beta W^{\top} W+t_kI\right\|_2\left\|((W^{\top})_{\mathcal{C}^k})^{\dagger}\right\|_2^2}{\rho} \left\|A^{\top}A+\beta W^{\top}W\right\|_2 \left\|\bu^k-\bu^{k+1}\right\|^2\\
&&-\left\|W^{\top} W\right\|_2\left\|((W^{\top})_{\mathcal{C}^k})^{\dagger}\right\|_2^2 \left\|A^{\top}A+\beta W^{\top}W\right\|_2\left\|\bu^k-\bu^{k+1}\right\|^2\\
&\ge&\frac{t_k}{4}\left\|\bu^{k+1}-\bu^k\right\|^2.
\end{eqnarray*}
where the inequality $(*)$ follows by Lemma \ref{lemmaguk+1uk}, and the last inequality follows from \eqref{equ:tk2}.

(ii) If $\rank((W^{\top})_{\mathcal{C}^k})=0$, then $\bz(\cdot,\by^k)={\bf 0}$, using \eqref{inequgu} and \eqref{inequgy} we have
\begin{eqnarray*}
&& h(\bu^k,\by^k)-h(\bu^{k+1},\by^{k+1})\\
&=& h(\bu^k,\by^k)-h(\bu^{k+1},\by^k)+h(\bu^{k+1},\by^k)-h(\bu^{k+1},\by^{k+1})\\
&=& f(\bu^k)+\lambda\left\|\by^k\right\|_0+\frac{\beta}{2}\left\|W\bu^k\right\|^2 +\frac{\rho}{2}\left\|W\bu^k-\by^k\right\|^2- f(\bu^{k+1})-\lambda\left\|\by^k\right\|_0 \\
&&-\frac{\beta}{2}\left\|W\bu^{k+1}\right\|^2-\frac{\rho}{2}\left\|W\bu^{k+1}-\by^k\right\|^2+h(\bu^{k+1},\by^k)- h(\bu^{k+1},\by^{k+1})\\
&\ge&\frac{t_k}{2}\left\|\bu^{k+1}-\bu^k\right\|^2.
\end{eqnarray*}

\end{proof}

Next, we give the main result of this subsection.
\begin{theorem}
Under the same assumptions as in Theorem \ref{solutionguy} and Lemma \ref{gfunbound}, let $\{\bu^k,\by^k\}$ be a bounded sequence generated by Algorithm \ref{algor:clmem}. Then, the accumulation of $\{\bu^k,\by^k\}$ is critical point of $h$.
\end{theorem}

\begin{proof}
Since the sequence $\{\bu^k,\by^k\}$ is bounded, let $(\bu^*,\by^*)$ be an accumulation point thereof. Then there exists a subsequence
$\{\bu^{k_t},\by^{k_t}\}$ such that $\lim_{t\to\infty} (\bu^{k_t},\by^{k_t})=(\bu^*,\by^*)$. By $\by^{k_t}\to \by^*$, when $t$ is sufficiently large, we have $\supp(\by^*)\subset \supp(\by^{k_t})$. We prove the conclusions by two cases.

(a) There exists $k_j$, When $t\ge j$, $\supp(\by^{k_t})=\supp(\by^*)$. Let $\supp(\by^*)=\ci^*$ and $\mathcal{C}^*=\{i|y_i^*=0\}$.

(i) First, we consider the case $\mathcal{C}^{k_t}=\cc^*\neq \varnothing$ when $t\ge j$.

By \eqref{cquesu} and \eqref{cquesy}, taking $k=k_t$ and $k=k_t-1$ respectively, we have
\begin{eqnarray}\label{gfinalfuzhu1}
&&A^{\top}(A\bu^{k_t+1}-\bu_0)+W^{\top}\bz(\bu^{k_t},\by^{k_t})+\beta W^{\top}W\bu^{k_t+1}+t_k(\bu^{k_t+1}-\bu^{k_t})\nonumber\\
&&+\rho W^{\top}(W\bu^{k_t+1}-\by^{k_t})={\bf 0}
\end{eqnarray}
and
\begin{eqnarray}\label{gfinalfuzhu}
\eta^{k_t}+\rho(\by^{k_t}-W\bu^{k_t})-\bz(\bu^{k_t},\by^{k_t})={\bf 0},
\end{eqnarray}
where $\eta^{k_t}\in \partial \lambda\|\by^{k_t}\|_0$.

When $i\in \ci^*$, by $\supp(\by^{k_t})=\supp(\by^*)=\ci^*$ and \eqref{gfinalfuzhu}, we obtain $0+\rho(\by^{k_t}-W\bu^{k_t})_i=0$. Let $t\to \infty$, we have $\rho(\by^*-W\bu^*)_{\mathcal{I}^*}={\bf 0}$.

Furthermore, when $t>j$, $\mathcal{C}^{k_t}=\mathcal{C}^{*}$, then $\lim_{t\to \infty}\bz(\bu^{k_t}, \by^{k_t})_{\mathcal{C}^{k_t}}=\lim_{t\to \infty}\bz(\bu^{k_t},$ $\by^{k_t})_{\mathcal{C}^{*}}=\lim_{t\to \infty}((W^{\top})_{\mathcal{C}^*})^{\dagger}(-A^{\top}(A\bu^{k_t}-\bu_0)-\beta W^{\top}W\bu^{k_t})=\bz(\bu^*,\by^*)_{\cc^*}$, $\bz(\bu^*,\by^*)_{\mathcal{I}^*}$ $={\bf 0}$. 
By Theorem \ref{solutionguy}, we have if $\by_i^{k_t}\neq 0$, $\by_i^{k_t}=(W\bu^{k_t})_i$.  Since Lemma \ref{gfunbound} and $\{(\bu^k,\by^k)\}$ is bounded, we have $\lim_{k\to\infty}\|\bu^{k+1}-\bu^k\|=0$.  Therefore, let $t\to \infty$,  by \eqref{gfinalfuzhu1} we have
\begin{align}\label{gfinalfuzhu3}
A^{\top}(A\bu^*-\bu_0)&+(W^{\top})_{\mathcal{C}^*}((W^{\top})_{\mathcal{C}^*})^{\dagger}(-A^{\top}(A\bu^*-\bu_0)-\beta W^{\top}W\bu^*)
 +\beta W^{\top}W\bu^*\nonumber\\
 & +\rho (W^{\top})_{\mathcal{C}^*} (W\bu^*-\by^*)_{\mathcal{C}^*}={\bf 0}
\end{align}
Splitting both the first and third terms of the above equation into parts, we have
\begin{eqnarray*}
&&(W^{\top})_{\mathcal{C}^*}((W^{\top})_{\mathcal{C}^*})^{\dagger} \left[A^{\top}(A\bu^*-\bu_0)+\beta W^{\top}W\bu^*\right]\\
&&+\left[I-(W^{\top})_{\mathcal{C}^*}((W^{\top})_{\mathcal{C}^*})^{\dagger}\right] \left[A^{\top}(A\bu^*-\bu_0)+\beta W^{\top}W\bu^*\right]\\
&&+(W^{\top})_{\mathcal{C}^*}((W^{\top})_{\mathcal{C}^*})^{\dagger} \left[-A^{\top}(A\bu^*-\bu_0)-\beta W^{\top}W\bu^*\right]+\rho (W^{\top})_{\mathcal{C}^*} (W\bu^*-\by^*)_{\mathcal{C}^*}={\bf 0}.
\end{eqnarray*}
Furthermore, 
\begin{eqnarray*}
\left[I-(W^{\top})_{\mathcal{C}^*}((W^{\top})_{\mathcal{C}^*})^{\dagger}\right] \left[A^{\top}(A\bu^*-\bu_0)+\beta W^{\top}W\bu^*\right]+\rho (W^{\top})_{\mathcal{C}^*} (W\bu^*-\by^*)_{\mathcal{C}^*}={\bf 0}.
\end{eqnarray*}
Since $\Range(I-(W^{\top})_{\cc^*}((W^{\top})_{\cc^*})^{\dagger}) \cap\Range((W^{\top})_{\cc^*})=\{\mathbf{0}\}$, we have 
\[
(I-(W^{\top})_{\mathcal{C}^*}((W^{\top})_{\mathcal{C}^*})^{\dagger})(A^{\top}(A\bu^*-\bu_0)+\beta W^{\top}W\bu^*)={\bf 0}
\]
and 
\BE\label{equ:last}
(W^{\top})_{\mathcal{C}^*} (W\bu^*-\by^*)_{\mathcal{C}^*}={\bf 0}.
\EE
Substituting $(W^{\top})_{\mathcal{C}^*} (W\bu^*-\by^*)_{\mathcal{C}^*}={\bf 0}$ into \eqref{gfinalfuzhu3}, we have
\BE\label{equ:kkt4}
A^{\top}(A\bu^*-\bu_0)+W^{\top}\bz^* +\beta W^{\top}W\bu^*={\bf 0}.
\EE
Moreover, by \eqref{equ:last}, we have
\begin{eqnarray*}
&& (W^{\top})_{\mathcal{C}^*} (W\bu^*)_{\cc^*}={\bf 0}
\Leftrightarrow  W^{\top}I_{\cc^*}I_{\cc^*}^{\top}(W\bu^*)={\bf 0}
\Leftrightarrow  (W^{\top}I_{\cc^*})(W^{\top}I_{\cc^*})^{\top}\bu^*={\bf 0}\\
&&\Leftrightarrow (W^{\top}I_{\cc^*})^{\top}\bu^*={\bf 0}
\Leftrightarrow I_{\cc^*}^{\top}W\bu^*={\bf 0}
\Leftrightarrow  I_{\cc^*}^{\top}(W\bu^*-\by^*)={\bf 0}
\Leftrightarrow (W\bu^*-\by^*)_{\cc^*}={\bf 0}.
\end{eqnarray*}
Therefore, 
\BE\label{equ:kkt5}
W\bu^*-\by^*={\bf 0}.
\EE
Invoking this result in \eqref{gfinalfuzhu} yields: when $i\in\cc^*$, there $\eta^*\in\lambda\|y_i^*\|_0$ such that
\BE\label{equ:kkt6}
\eta^*-\bz^*={\bf 0}.
\EE
Combining equations \eqref{equ:kkt4}--\eqref{equ:kkt6}, it follows that $(\bu^*,\by^*)$ is a KKT point of \eqref{equ:reformu}.

(ii)when $t>j$, ${\mathcal{C}^{k_t}}=\cc^*=\varnothing$, then $\bz^{k_t}={\bf 0}$ for all $t>j$.
By \eqref{cquesu} and \eqref{cquesy}, taking $k=k_t$ and $k=k_t-1$ respectively, we also have
\begin{equation}\label{gfinalfuzhu1empty}
A^{\top}(A\bu^{k_t+1}-\bu_0)+\beta W^{\top}W\bu^{k_t+1}+t_k(\bu^{k_t+1}-\bu^{k_t})+\rho W^{\top}(W\bu^{k_t+1}-\by^{k_t})={\bf 0}
\end{equation}
and
\begin{equation}\label{gfinalfuzhuempty}
\eta^{k_t}+\rho(\by^{k_t}-W\bu^{k_t})={\bf 0},
\end{equation}
where $\eta^{k_t}\in \partial \lambda\|\by^{k_t}\|_0$.
By ${\mathcal{C}^{k_t}}=\cc^*=\varnothing$, we obtain $\eta^{k_t}={\bf 0}$ and $\by^{k_t}-W\bu^{k_t}=0$. Using the boundedness of sequence $\{(\bu^k,\by^k)\}$ together with Lemma \ref{gfunbound}, let $t\to \infty$ in \eqref{gfinalfuzhu1empty} and \eqref{gfinalfuzhuempty}, we have
\[
A^{\top}(A\bu^*-\bu_0)+W^{\top}\bz^*+\beta W^{\top}W\bu^*={\bf 0},\quad W\bu^*-\by^*={\bf 0} \quad \mbox{and} \quad \eta^*-\bz^*={\bf 0}.
\]
where $\eta^*={\bf 0} \in \partial \lambda\|\by^{k_t}\|_0$. This proves that $(\bu^*, \by^*)$ is a KKT point of Problem \eqref{equ:reformu}.

(b) $\supp(\by^*)\subset \supp(\by^{k_t})$ and $\supp(\by^*)\neq\supp(\by^{k_t})$ for sufficiently large $t$.

By Theorem  \ref{solutionguy}, for all $k\ge 0$ and $i\in [m]$, we have $y_i^{k+1}=(W\bu^{k+1})_i$ if $|(W\bu^{k+1})_i|$ exceeds a fixed positive constant, and $y_i^{k+1}=0$ otherwise. This indicates that the sequence formed by the components of $\{\by^{k_t}\}$ can consist of entirely nonzero entries for all sufficiently large t while converging to zero. Equivalently, for $\by$ updated via \eqref{equ:yk}, the support identity $\supp(\by^*)=\supp(\by^{k_t})$ must hold. 

Therefore, the second scenario cannot occur.
\end{proof}

\section{Numerical experiments}\label{sec:4}
In this section, we report some numerical experiments on the  performance of  ALMEMfrr (Algorithm \ref{algor:rlmem}) and  ALMEMgc (Algorithm \ref{algor:clmem}) for solving problem (1.1).
To illustrate the effectiveness of our method, we compare the proposed algorithm with the inexact augmented Lagrangian method (IALM) in \cite{SL20}. Our numerical tests were implemented by running {\tt MATLAB R2024b} on a personal desktop with the processor Intel Core i9-14900K @ 3.2GHz and 192 GB RAM.

For comparison purposes, let
\begin{eqnarray*}
({\tt kktL(\bu^k,\by^k)})^2&=&\|A^{\top}(A\bu^k-\bu_0)+W^{\top}\bz^k+\beta W^{\top}W\bu^k
+\rho W^{\top}(W\bu^k-\by)\|^2\\
&&+\|(\bz^{k}+\rho(W\bu^k-\by^k))_{{\mathcal{I}}^k}\|^2,\\
({\tt kkt(\bu^k,\by^k)})^2&=&\|\bz_{{\mathcal{I}}^k}^k\|^2+\|A^{\top}(A\bu^k-\bu_0)+W^{\top}\bz^k+\beta W^{\top}W\bu^k)\|^2,\\
{\tt obj(\bu^k,\by^k)}&=&f(\bu^k)+\lambda\|\by^k\|_0 +\frac{\beta}{2}\|W\bu^k\|^2.
\end{eqnarray*}
In our numerical tests, the stopping criterion is set to be ${\tt kktL}(\bu^k,\by^k) \le {\tt tol}$
and the largest number of iterations  is set to be {\tt ITmax}, where `{\tt tol}' is a prescribed tolerance. Let  `{\tt ct.}' and `{\tt rce.}'  denote the total computing time in seconds and the relative constraint error $\|W\bu^k-\by^k\|/\|W\|$ at the final iterates of the corresponding algorithms, respectively. We choose $t_k$ using the Barzilai-Borwein (BB) strategy \cite{BB88}.

\subsection{Synthetic data}
In this subsection, owing to the simplicity and efficiency of the alternating minimization method and the close relation between
\BE\label{equ:alter-min}
\min_{\bu, \by} f(\bu)+\lambda\|\by\|_0 +\frac{\beta}{2}\|W\bu\|^2+\frac{\rho}{2}\|W\bu-\by\|^2
\EE
and problem \eqref{intro}, we employ the alternating minimization method to solve \eqref{equ:alter-min} and use the obtained solution as the initial point. All algorithms are stopped if the running time exceeds 7200 seconds or the number of iterations reaches {\tt ITmax}.

We first examine the performance of the proposed algorithms when $W$ is of full rank.

\begin{example}\label{synfullrank}
We consider problem \eqref{intro} with $m=800$, $n=1000$.
The matrix $W = [\hat{W}, \tilde{W}]\in\mathbb{R}^{m\times n}$ is randomly generated via the {\tt randn} function, where $\hat{W}\in\mathbb{R}^{m\times m}$ and $\tilde{W}\in\mathbb{R}^{m\times (n-m)}$. We also set matrix $A$ as the identity matrix. The vector $\mathbf{y}\in\mathbb{R}^{m}$ is sparse random with $\|\mathbf{y}\|_0 = 80$.
To construct the $\mathbf{u}_0\in\mathbb{R}^{n}$, we first randomly generate $\tilde{\mathbf{u}}_0\in\mathbb{R}^{n-m}$, and set $\bar{\mathbf{u}}_0 = \hat{W}^{\dagger}(\mathbf{y} - \tilde{W} \tilde{\mathbf{u}}_0)$. The vector $\bu_0$ is then constructed as $\mathbf{u}_0 = [\bar{\mathbf{u}}_0^{\top}, \tilde{\mathbf{u}}_0^{\top}]^{\top} + 0.001\cdot{\tt randn}(n,1)$ by adding a small Gaussian perturbation.
Finally, we set the hyperparameters as $\beta = \lambda = 0.1$.
\end{example}

\begin{figure}[!ht]
	\centering
	\includegraphics[width=0.4\textwidth,height=0.18\textheight]{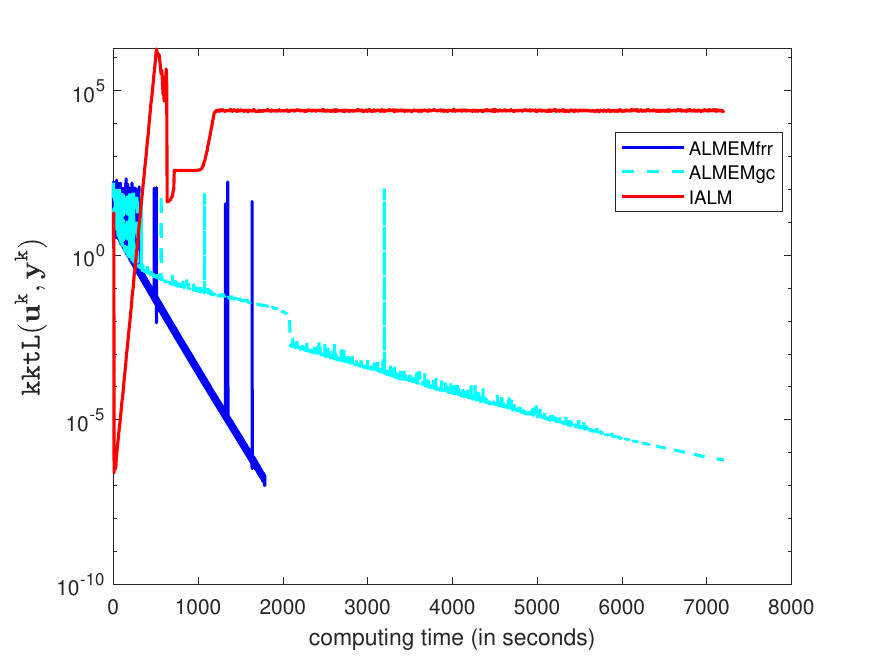}
	\includegraphics[width=0.4\textwidth,height=0.18\textheight]{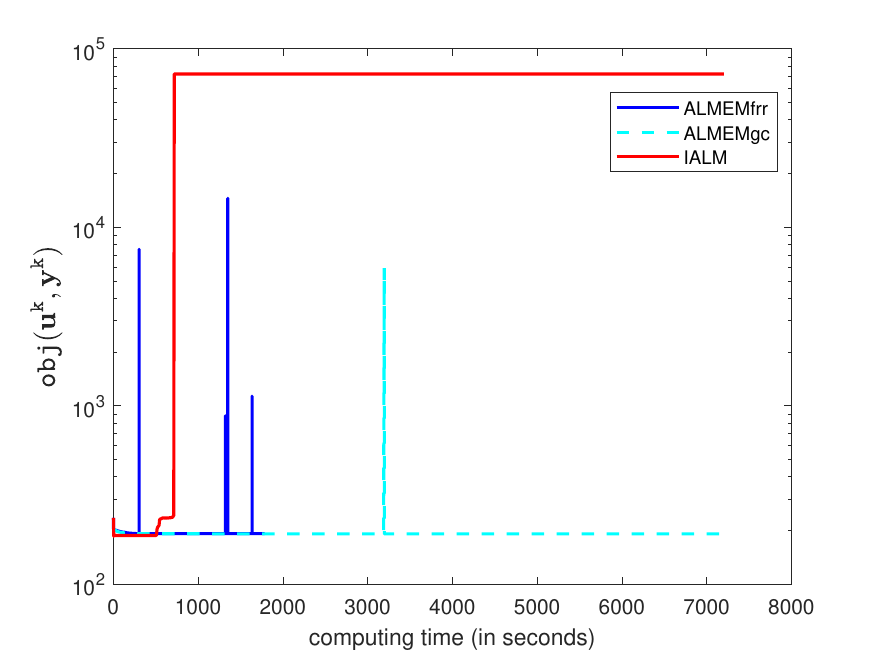}
	\includegraphics[width=0.4\textwidth,height=0.18\textheight]{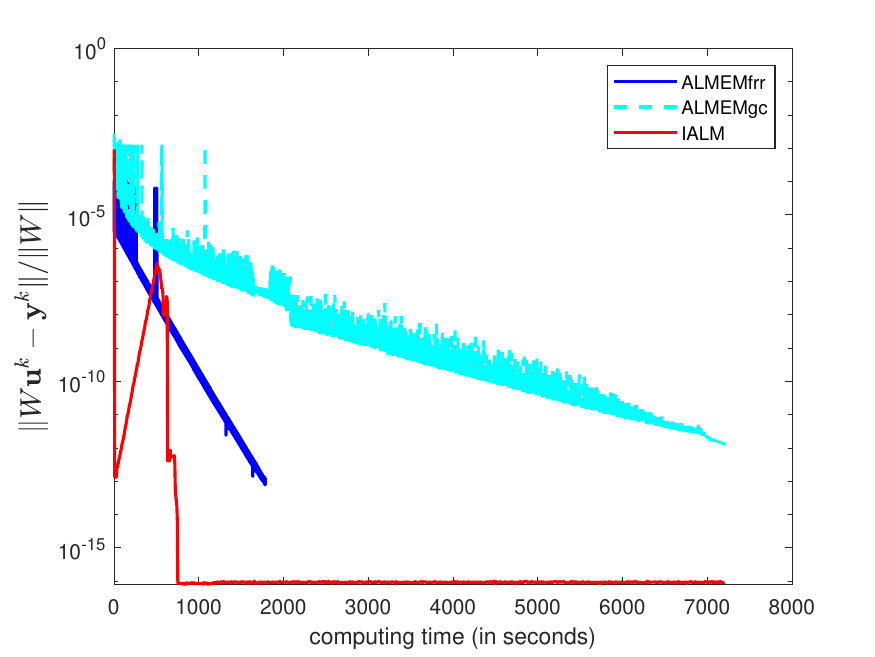}
	\caption{Convergence and computing time curves for Example  \ref{synfullrank}.}\label{figuresynrowfull}
\end{figure}

\begin{table}[!ht]\renewcommand{\arraystretch}{1.0} \addtolength{\tabcolsep}{1.0pt}
	\begin{center}{\footnotesize
			\begin{tabular}[c]{|c|c|c|c|c|c|}     \hline
				Alg. & ${\tt obj(\bu^k,\by^k)}$ &  {\tt ct.}   & {\tt rce.} &${\tt kktL(\bu^k,\by^k)}$&${\tt kkt(\bu^k,\by^k)}$ \\ \hline
				
				ALMEMfrr                    & $1.9238\times 10^2$  &$1.7837\times 10^3$  &$7.8327\times 10^{-14}$&$9.5777\times 10^{-8}$ &$1.4116\times 10^{-8}$ \\ \hline
				ALMEMgc                     & $1.9198\times 10^2$  &$7.2000\times 10^3$   &$1.4094\times 10^{-12}$ &$6.0541\times 10^{-7}$&$6.0190\times 10^{-7}$ \\ \hline
				{IALM}                           &$7.2419\times 10^4$    & $7.2002\times10^3$    &$9.3461\times 10^{-17}$ &$2.4961\times 10^{4}$&$1.1581\times 10^4$\\ \hline
		\end{tabular}}
	\end{center}
	\caption{Numerical results for Example \ref{synfullrank}.} \label{Table:synfullrank}
\end{table}

Figure \ref{figuresynrowfull} and Table \ref{Table:synfullrank} display the numerical results of all comparative algorithms for Example \ref{synfullrank}, where we set $\rho=100$ for ALMEMfrr, $\rho=50$ for ALMEMgc, with a stopping tolerance ${\tt tol} = 10^{-7}$ and maximum iteration number ${\tt ITmax} = 10^6$. It can be clearly observed from the presented results that both ALMEMfrr and ALMEMgc achieve stable and successful convergece. In particular, ALMEMfrr converges accurately under strict tolerance conditions and consumes significantly less computational time than ALMEMgc, demonstrating its higher computational efficiency. In contrast, the IALM algorithm exhibits unstable iterative performance. Although the ${\tt kktL(\bu^k,\by^k)}$ value decreases rapidly at the initial stage, it increases sharply in subsequent iterations and finally fails to converge to a feasible KKT point of the problem \eqref{intro}.

Next, we consider the performance of the proposed algorithms on synthetic examples where $W$ is not of full rank.

\begin{example}\label{syngeneral}
	We consider problem \eqref{intro} with $m=800$, $n=1000$, and $r=600$. We first generate $\hat{W} = U\Sigma V^{\top} \in \mathbb{R}^{m\times n}$ using the {\tt randn} function, and take $W = U_r \Sigma_r V_r^{\top}$. Generate $A \in \mathbb{R}^{m\times n}$ independently via {\tt randn}.
	Let $\mathbf{y} \in \mathbb{R}^{m}$ be a sparse random vector with $\|\mathbf{y}\|_0 = 80$.
	Let $\mathbf{u}_0 = A W^{\dagger} \mathbf{y} + \epsilon \boldsymbol{\theta}$, where $\boldsymbol{\theta} = {\tt randn}(m,1)$ and $\epsilon = 0.001$.
	The regularization parameters are set to $\beta = \lambda = 0.1$.
\end{example}

\begin{figure}[!ht]
	\centering
	\includegraphics[width=0.4\textwidth,height=0.18\textheight]{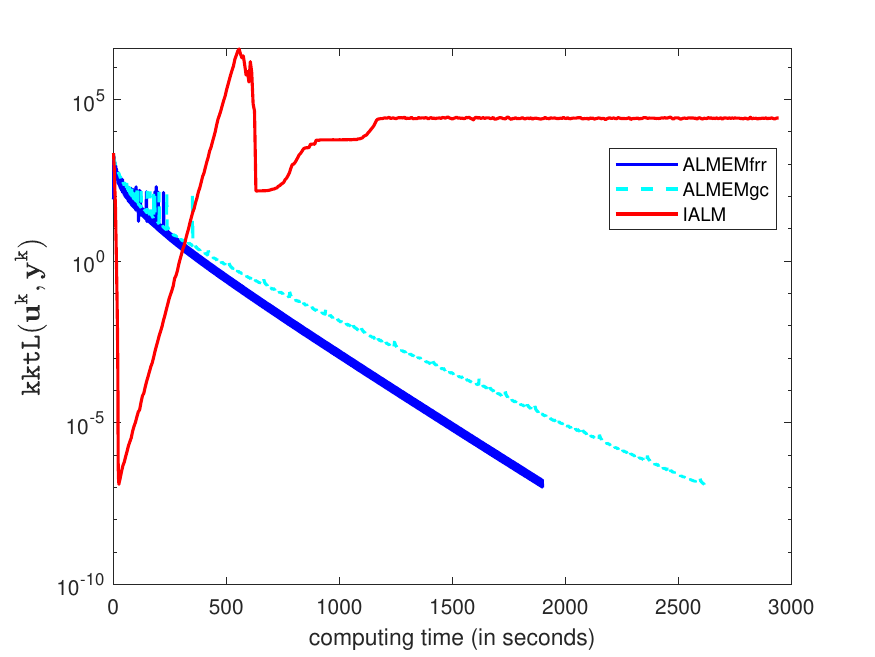}
	\includegraphics[width=0.4\textwidth,height=0.18\textheight]{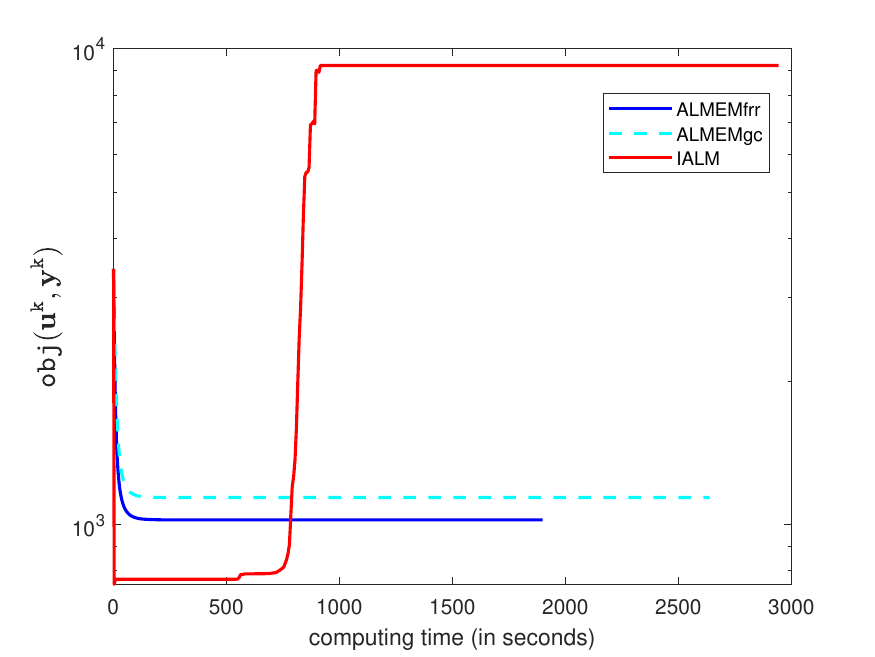}
	\includegraphics[width=0.4\textwidth,height=0.18\textheight]{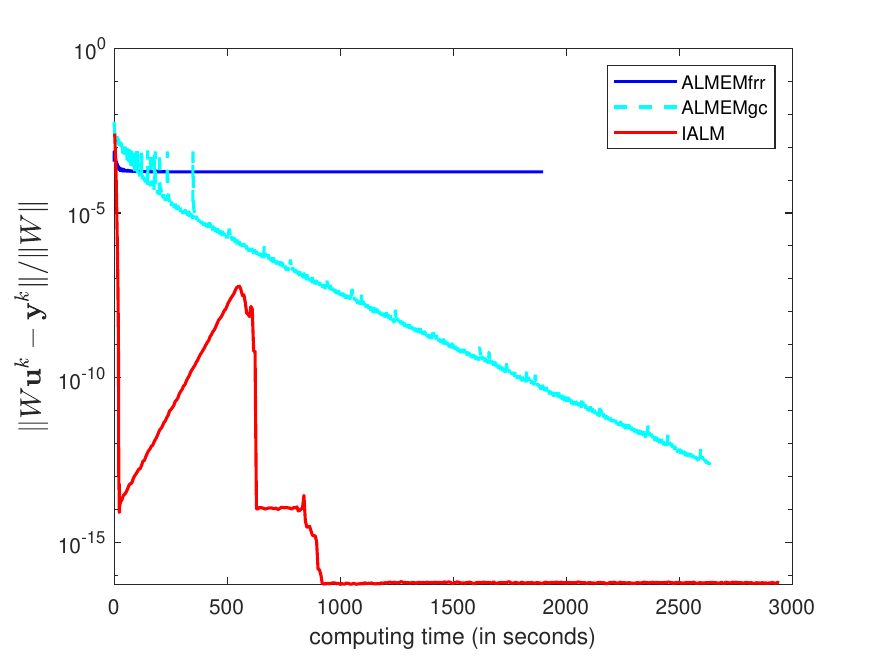}
	\caption{Convergence and computing time curves for Example  \ref{syngeneral}.}\label{figuresyngeneral}
\end{figure}

\begin{table}[!ht]\renewcommand{\arraystretch}{1.0} \addtolength{\tabcolsep}{1.0pt}
	\begin{center}{\footnotesize
			\begin{tabular}[c]{|c|c|c|c|c|c|}     \hline
				Alg. & ${\tt obj(\bu^k,\by^k)}$ &  {\tt ct.}   & {\tt rce.} &${\tt kktL(\bu^k,\by^k)}$&${\tt kkt(\bu^k,\by^k)}$ \\ \hline
				
				ALMEMfrr                    & $1.0239\times 10^3$  &$1.8982\times 10^3$  &$1.7830\times 10^{-4}$&$9.9555\times 10^{-8}$ &$8.9324\times 10^0$ \\ \hline
				ALMEMgc                     & $1.1409\times 10^3$  &$2.6373\times 10^3$   &$2.1732\times 10^{-13}$ &$9.9803\times 10^{-8}$ & $9.5782\times 10^{-8}$ \\ \hline
				{IALM}                           &$9.2293\times 10^3$    & $7.2018\times 10^3$    &$6.0459\times 10^{-17}$ &$2.7842\times 10^{4}$ &$1.2604\times 10^4$\\ \hline
		\end{tabular}}
	\end{center}
	\caption{Numerical results for Example \ref{syngeneral}.} \label{Table:syngeneral}
\end{table}

Figure \ref{figuresyngeneral} and Table \ref{Table:syngeneral} report the numerical results of Example \ref{syngeneral}. In this numerical test, we set $\rho=100$ for both ALMEMfrr and ALMEMgc, and adopt a stopping tolerance ${\tt tol} = 10^{-7}$ together with the maximum iteration number ${\tt ITmax} = 10^6$. As clearly reflected by the experimental data, both ALMEMfrr and ALMEMgc successfully converge to a KKT point of the optimization problem $\min h(\bu,\by)$, and  ALMEMfrr evidently outperforms its counterpart in computational efficiency with shorter running time. However, there exists a notable difference between the two methods: the limit point yielded by ALMEMfrr does not satisfy the KKT conditions of the original optimization problem, while the limit point from ALMEMgc fully complies with the relevant conditions. This phenomenon is well consistent with the theoretical results derived in our paper. Additionally, the iterative sequence generated by IALM cannot converge to a KKT point of $\min L(\bu,\by,\bz)$ with satisfactory precision throughout the whole iteration process.

\subsection{Real data}
﻿
We first consider the trend filtering problem, which originates from \cite{KK09, YG18}.

\begin{example}\label{truedatadapan}
	Given a time series $\bu_0$, the objective is to find another time series $\bu$ that is closest to $\bu_0$, while ensuring that $\bu$ exhibits sparsity after a gradient mapping. We performed trend filtering on the `snp500.dat' dataset, where the experimental data consist of 300 consecutive trading days starting from March 25, 1999. $W\in \R^{(n-2) \times n}$ is a $2^{nd}$ difference matrix. We set $\beta=1$ and $\lambda=100$.
\end{example}

\begin{figure}[!ht]
	\centering
	\includegraphics[width=0.4\textwidth,height=0.18\textheight]{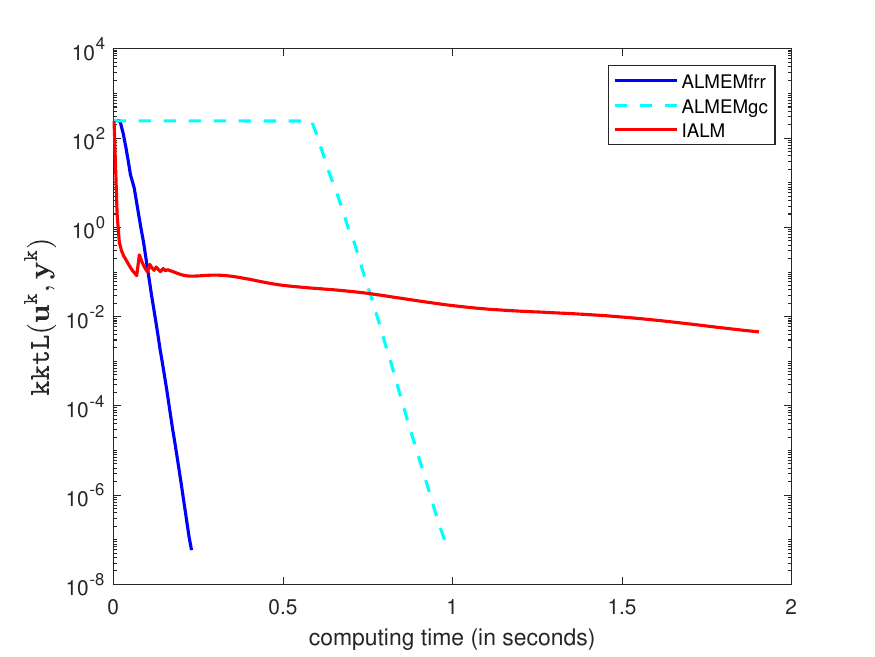}
	\includegraphics[width=0.4\textwidth,height=0.18\textheight]{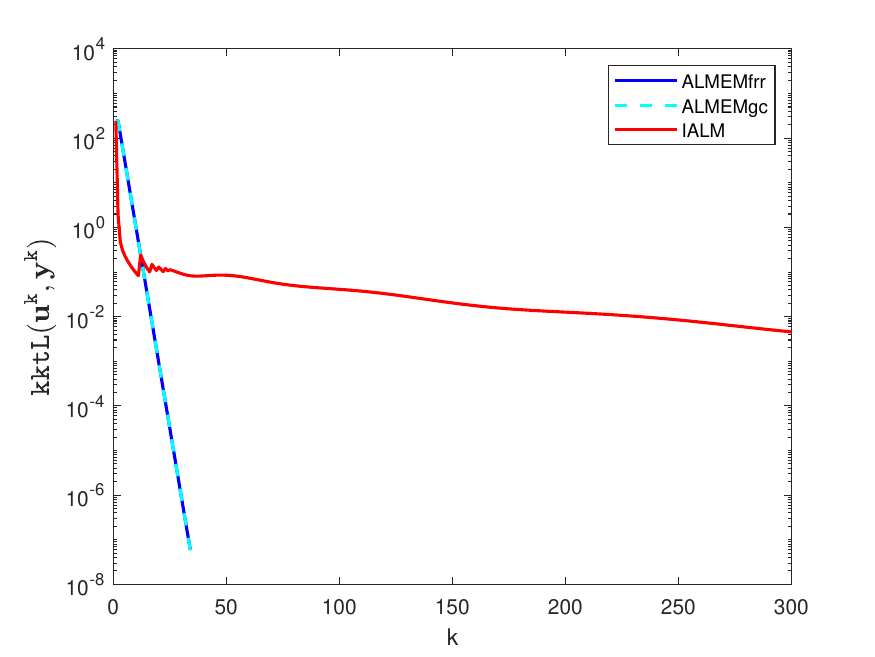}
	\includegraphics[width=0.4\textwidth,height=0.18\textheight]{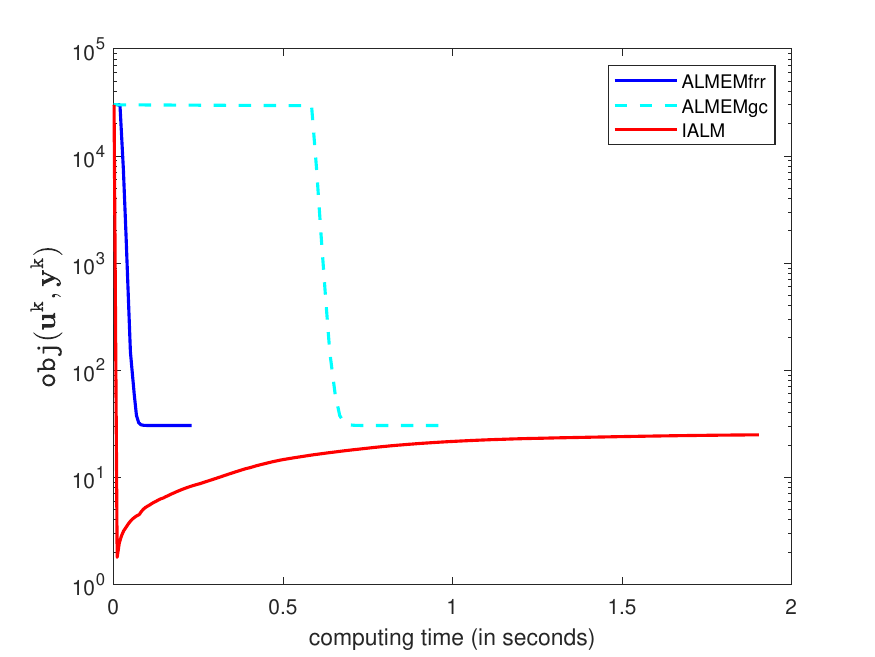}
	\includegraphics[width=0.4\textwidth,height=0.18\textheight]{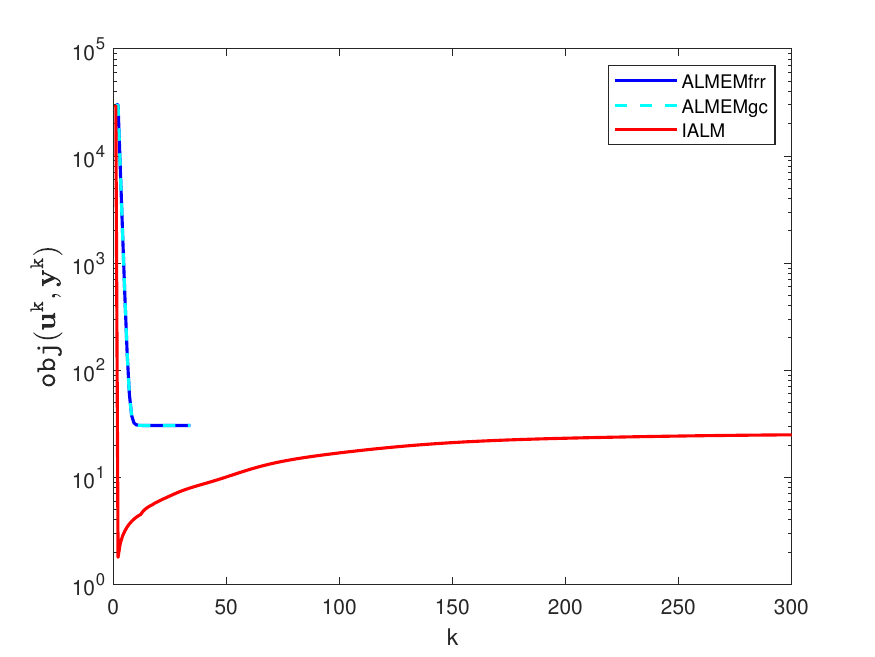}
	\includegraphics[width=0.4\textwidth,height=0.18\textheight]{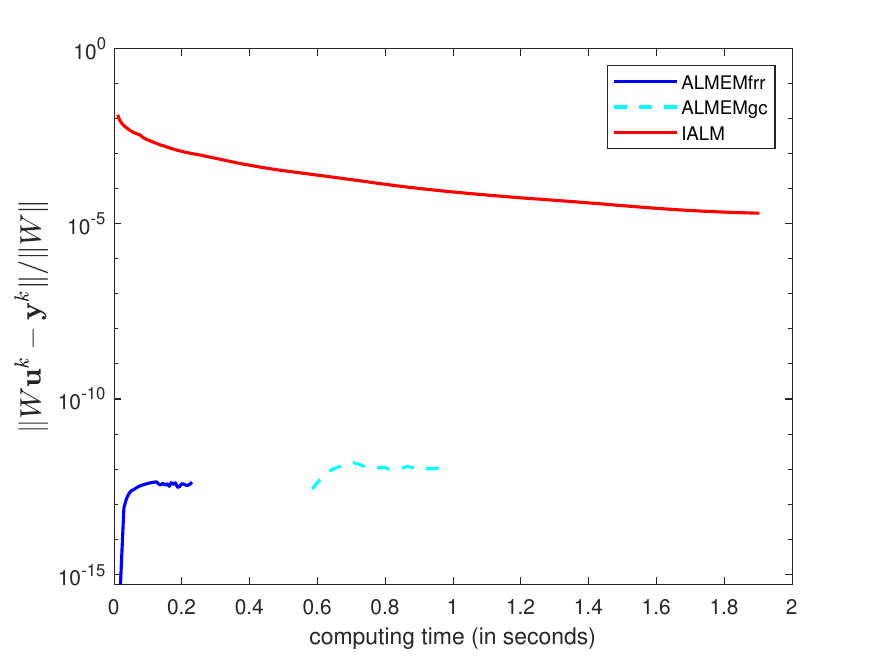}
	\includegraphics[width=0.4\textwidth,height=0.18\textheight]{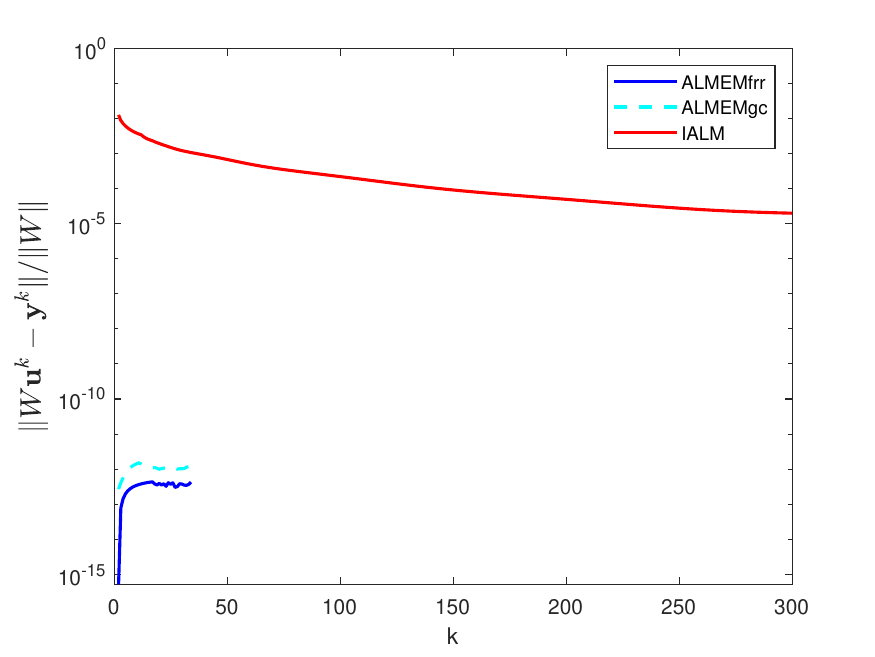}
	\caption{Success rate of recovery for Example  \ref{truedatadapan}.}\label{truefigdapan}
\end{figure}

\begin{table}[!ht]\renewcommand{\arraystretch}{1.0} \addtolength{\tabcolsep}{1.0pt}
	\begin{center}{\footnotesize
			\begin{tabular}[c]{|c|c|c|c|c|c|}     \hline
				Alg. & ${\tt obj(\bu^k,\by^k)}$ &  {\tt ct.}   & {\tt rce.} &${\tt kktL(\bu^k,\by^k)}$&${\tt kkt(\bu^k,\by^k)}$ \\ \hline
				ALMEMfrr                    & $3.0477\times 10^1$  &$2.2997\times 10^{-1}$  &$4.2603\times 10^{-13}$&$5.8739\times 10^{-8}$ &$1.1470\times 10^{-7}$ \\ \hline
				ALMEMgc                     & $3.0477\times 10^1$  &$9.8548\times 10^{-1}$   &$1.0830\times 10^{-12}$ &$5.6549\times 10^{-8}$ & $1.1308\times 10^{-7}$ \\ \hline
				{IALM}                           &$3.0477\times 10^1$    & $7.2012\times 10^3$    &$1.4989\times 10^{-11}$ &$4.9983\times 10^{-2}$ &$2.2670\times 10^{-2}$\\ \hline
		\end{tabular}}
	\end{center}
	\caption{Numerical results for Example \ref{truedatadapan}.} \label{Table:truedatadapan}
\end{table}

Figure \ref {truefigdapan} and Table \ref {Table:truedatadapan} present the numerical results of the real-data experiment in Example \ref {truedatadapan}. In this test, we set $\rho=5000$ for ALMEMfrr and $\rho=100$ for ALMEMgc, with a unified stopping tolerance ${\tt tol} = 10^{-7}$ and a maximum iteration number ${\tt ITmax} = 10^6$. The numerical results demonstrate that both ALMEMgc and ALMEMfrr can accurately converge to valid KKT points of the original optimization problem. In contrast, the conventional IALM method only achieves a KKT residual of approximately $10^{-2}$ after a long runtime of 7200 seconds, failing to obtain a high-precision solution. In terms of computational efficiency, ALMEMfrr consumes the shortest running time, followed by ALMEMgc, while the IALM algorithm suffers from significantly higher computational cost. Moreover, the iterative curve in Figure \ref{truefigdapan} verifies that ALMEMfrr exhibits stable linear convergence behavior, which is perfectly consistent with the theoretical convergence results established in this paper.

\begin{example}[Iamge smoothing]\label{truedatapicture}
	We evaluate the performance of the proposed algorithm on the image smoothing. In this example, the image is sourced from \cite{FY18}, and we employ a difference operator that considers the vertical direction. We set $\beta = \lambda = 2.5$.
\end{example}

\begin{figure}[!ht]
	\centering
	\includegraphics[width=0.4\textwidth,height=0.22\textheight]{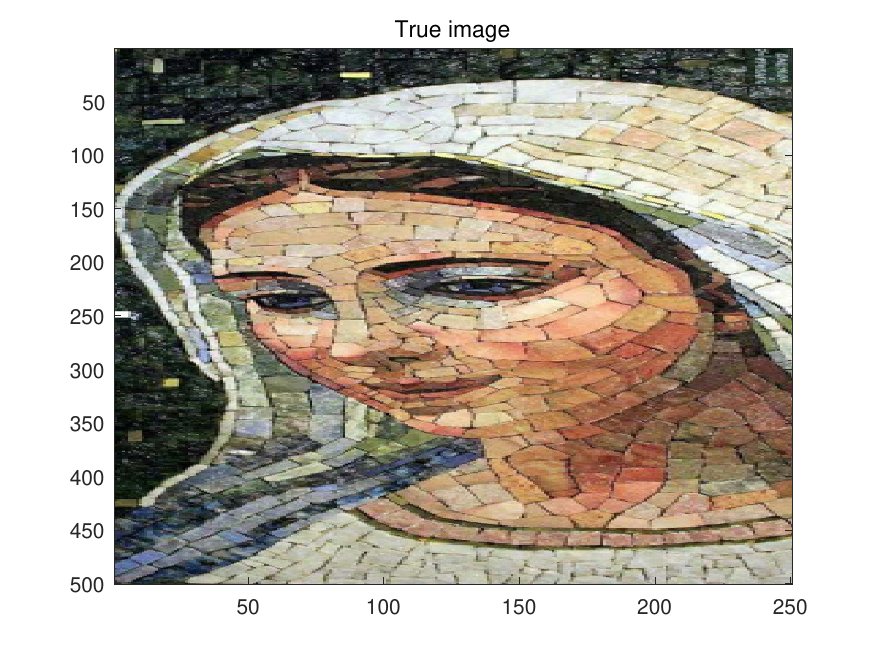}
	\includegraphics[width=0.4\textwidth,height=0.22\textheight]{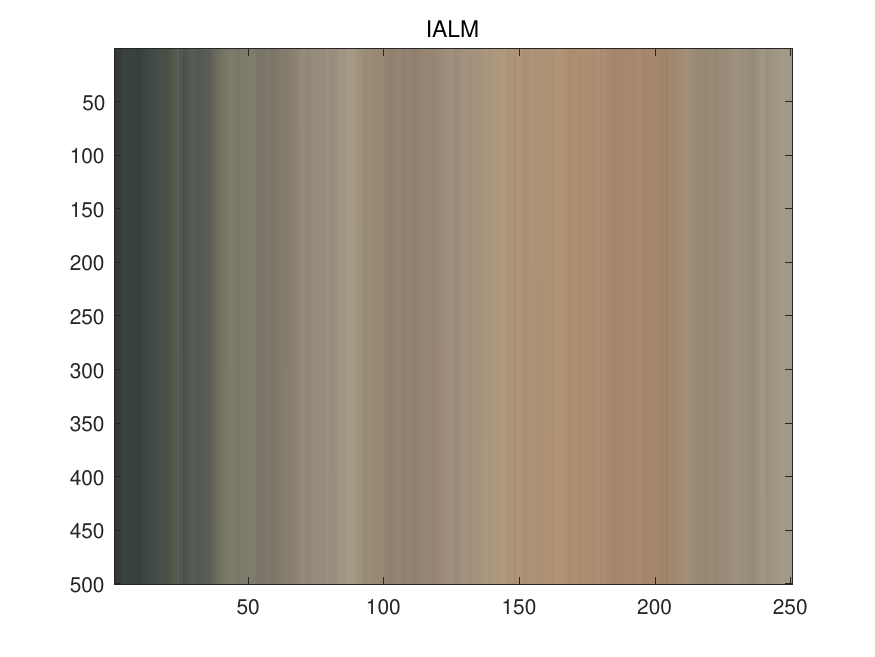}
	\includegraphics[width=0.4\textwidth,height=0.22\textheight]{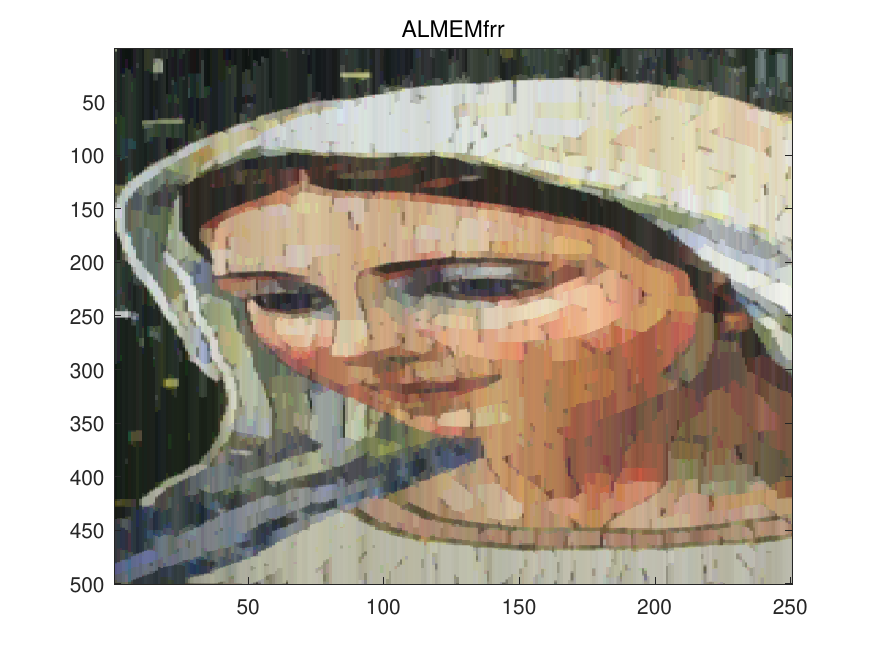}
	\includegraphics[width=0.4\textwidth,height=0.22\textheight]{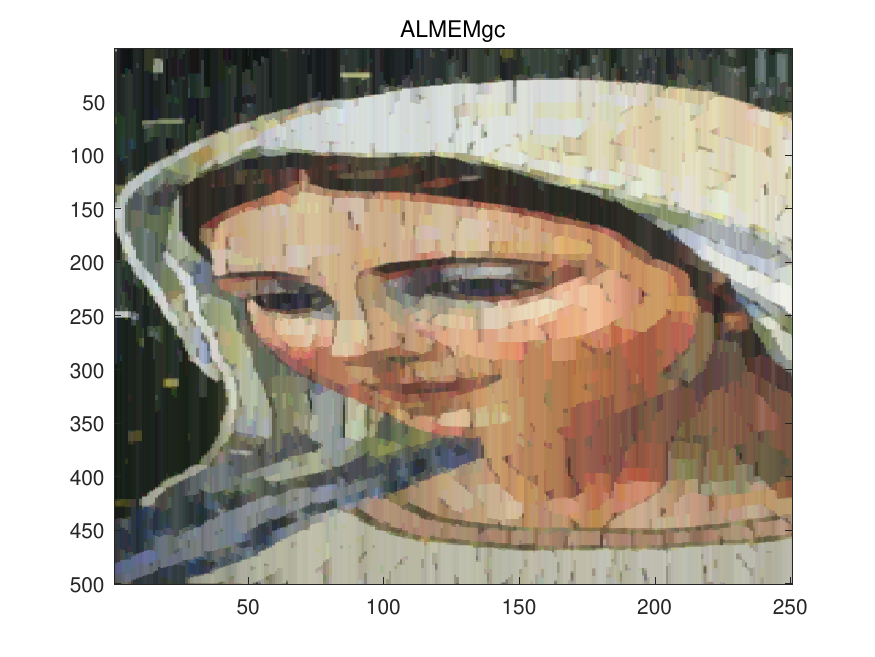}
	\caption{Constructed images via ALMEMfrr, ALMEMgc and IALM for Example  \ref{truedatapicture}.}\label{truefigpicture}
\end{figure}

Image smoothing suppresses undesired textures and noise while preserving prominent structural features, serving as a fundamental preprocessing step for salient object detection, image segmentation, image restoration and other computer vision tasks. In this part, we set $W$ as the gradient operator, i.e., $W\in\mathbb{R}^{(m-1)\times m}$ with $W_{ii}=-1, W_{i,i+1}=1$ and the other entries are zeros. We further test the texture removal task within the framework of image smoothing. Figure \ref {truefigpicture} displays the visual results of the real-image test in Example \ref{truedatapicture} with $\rho=100$ for ALMEMfrr and ALMEMgc, ${\tt tol} = 10^{-3}$ and ${\tt ITmax} = 10^6$. Additionally, we impose a runtime limit of 3600 seconds, where all algorithms are forcibly terminated once the elapsed time exceeds this threshold.

As observed in Figure \ref{truefigpicture}, both the proposed ALMEMgc and ALMEMfrr methods achieve favorable image restoration performance compared with the ground-truth image. They effectively eliminate undesired horizontal streaks while preserving the primary structural information of the original image. This satisfactory restoration performance validates the rationality of adopting only the vertical difference operator in our model design. By comparison, the baseline IALM method cannot achieve effective image restoration under the same experimental settings. It fails to remove image artifacts and inevitably results in severe loss of valid image details and original texture content.

\begin{example}[Image smoothing]\label{truedatapictureman}
	We evaluate the performance of the proposed algorithm on the image smoothing. the image is also sourced from \cite{FY18}, and we also employ a difference operator that considers the vertical direction. We set $\beta = \lambda = 0.8$.
\end{example}

\begin{figure}[!ht]
	\centering
	\includegraphics[width=0.4\textwidth,height=0.22\textheight]{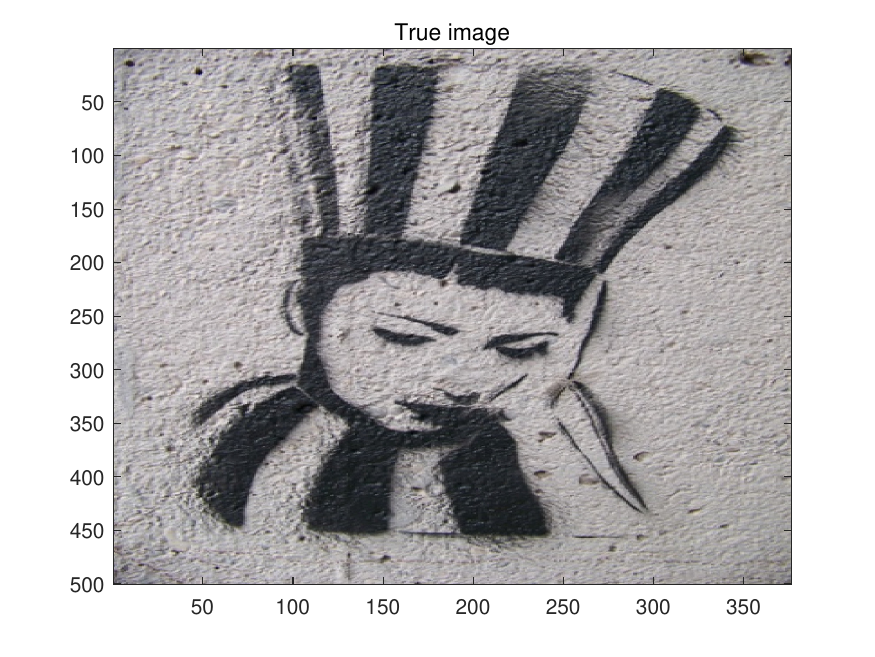}
	\includegraphics[width=0.4\textwidth,height=0.22\textheight]{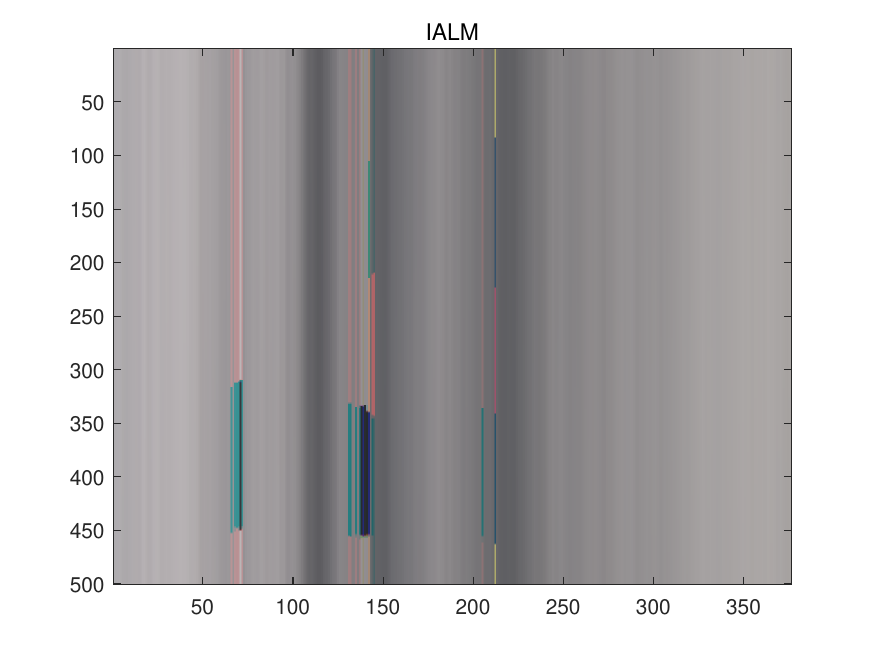}
	\includegraphics[width=0.4\textwidth,height=0.22\textheight]{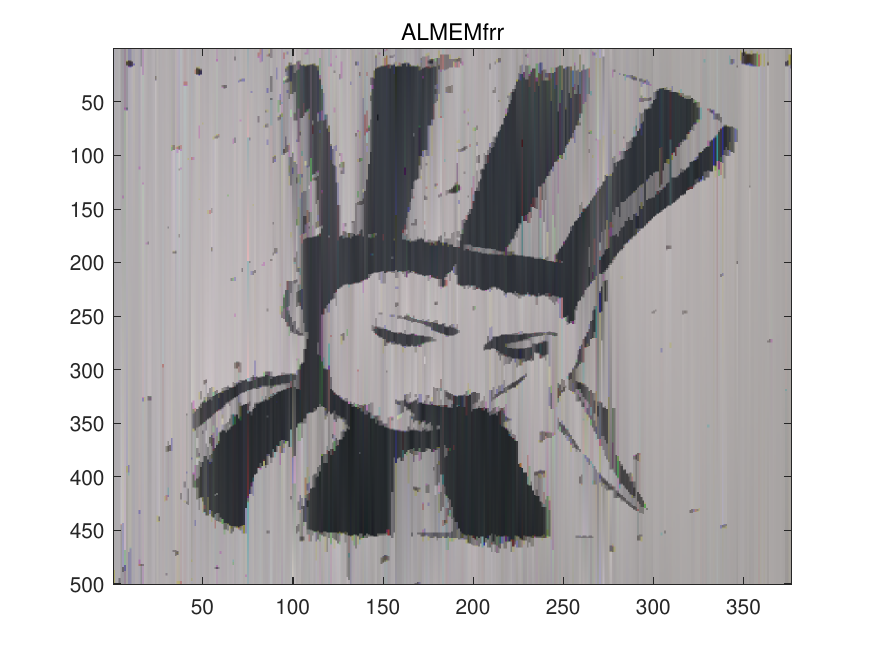}
	\includegraphics[width=0.4\textwidth,height=0.22\textheight]{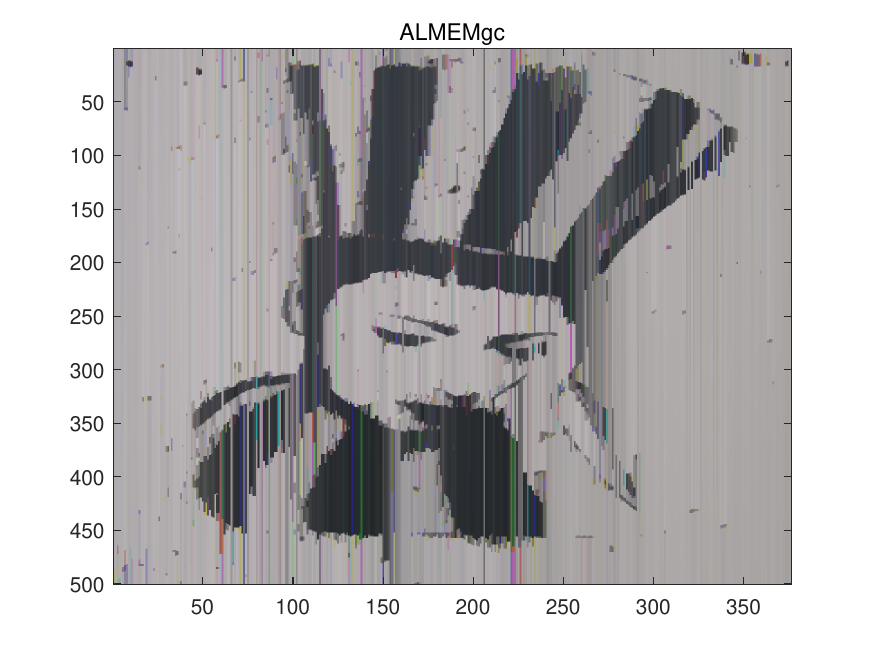}
	\caption{Constructed images via ALMEMfrr, ALMEMgc and IALM for Example  \ref{truedatapictureman}.}\label{truefigpictureman}
\end{figure}

In this example, we consider the case where $W$ is not of full rank. Specifically, $W$ is a row circulant matrix of $(1,-1,0,\ldots,0)$. Such circulant difference matrices are widely adopted for constructing smoothing priors in image processing tasks. Figure \ref{truefigpictureman} reports the corresponding numerical results of Example \ref{truedatapictureman} with $\rho=100$ for ALMEMfrr and ALMEMgc, ${\tt tol} = 10^{-3}$ and ${\tt ITmax} = 10^6$. To avoid excessive computation, we additionally set a runtime limit: all algorithms will stop execution once the elapsed time reaches 7200 seconds. 

As can be clearly seen from Figure \ref{truefigpictureman}, compared with the original ground-truth image, the proposed ALMEMfrr and ALMEMgc are capable of performing effective image smoothing and successfully eliminating various undesirable artifacts. In sharp contrast, the output generated by the IALM algorithm cannot maintain basic image characteristics, leading to obvious loss of original visual content and detailed information.

\section{Concluding Remarks}\label{sec:5}
This paper investigates a non-separable composite $\ell_0$-$\ell_2$  regularization model for inverse problems. We propose two exact augmented Lagrangian algorithms, ALMEMfrr and ALMEMgc, for full row-rank and arbitrary-rank transform matrices, respectively. Closed-form solutions are derived for all subproblems. Under mild assumptions, we prove that the sequence generated by ALMEMfrr converges linearly to a KKT point, and that any limit point of the sequence generated by ALMEMgc is a KKT point.

\vspace{2mm}
{\bf Conflict of Interest Statement} The authors declare that they have no conflict of interest.

{\bf Data Availability Statement} All data generated or analysed during this study are included in this manuscript.

\end{document}